\documentclass[11pt,reqno]{article}
\usepackage{graphicx}
\usepackage{amssymb}
\usepackage{amsmath}
\usepackage{amsthm}
\usepackage{amsfonts}
\usepackage{url}
\usepackage{hyperref}
\usepackage{breakurl}
\usepackage[T1]{fontenc}        

\newcommand{\comment}[1]{}
\newcommand{\raisecomma}{\raisebox{2pt}{$,$}}
\newcommand{\raisedot}{\raisebox{2pt}{$.$}}
\newcommand{\sign}{\text{sgn}}

\newcommand{\Dbar}{{\mathcal R}} 
\newcommand{\Had}{{\mathcal H}}

\newcommand{\R}{{\mathbb R}}

\newcommand{\Z}{{\mathbb Z}}

\newcommand{\Prob}{{\mathbb P}} 
\newcommand{\E}{{\mathbb E}}    
\newcommand{\V}{{\mathbb V}}    

\newcommand{\ve}{\varepsilon}

\newtheorem{theorem}{Theorem}
\newtheorem{corollary}{Corollary}
\newtheorem{lemma}{Lemma}
\newtheorem{proposition}{Proposition}

\newtheorem{remark}{Remark}

\begin{document}
\bibliographystyle{plain}
\title{~\\[-60pt]
Lower bounds on maximal determinants\\
 of binary matrices\\ 	
 via the probabilistic method
}
\author{Richard P.\ Brent\\
Australian National University\\
Canberra, ACT 0200,
Australia\\
\and
Judy-anne H.\ Osborn\\
The University of Newcastle\\
Callaghan, NSW 2308,
Australia\\
\and
Warren D.\ Smith\\
Center for Range Voting\\			
21 Shore Oaks Drive, Stony Brook\\ NY 11790, USA\\
}

\date{\today}

\maketitle
\thispagestyle{empty}                   

\begin{abstract}
Let $D(n)$ be the maximal determinant for $n \times n$ $\{\pm 1\}$-matrices,
and $\Dbar(n) = D(n)/n^{n/2}$ be the ratio of $D(n)$ to the Hadamard
upper bound.
We give several new lower bounds on $\Dbar(n)$ in terms of $d$, where
$n = h+d$, 
$h$ is the order of a Hadamard matrix,
and $h$ is maximal subject to $h \le n$. 
A relatively simple bound is
\[
\Dbar(n) \ge \left(\frac{2}{\pi e}\right)^{d/2}
	\left(1 - d^2\left(\frac{\pi}{2h}\right)^{1/2}\right)
 \;\text{ for all }\; n \ge 1.
\]
An asymptotically sharper bound is
\[
\Dbar(n) \ge \left(\frac{2}{\pi e}\right)^{d/2}
	\exp\left(d\left(\frac{\pi}{2h}\right)^{1/2} +
	\; O\left(\frac{d^{5/3}}{h^{2/3}}\right)\right).
\]
We also show that
\[
\Dbar(n) \ge \left(\frac{2}{\pi e}\right)^{d/2}
\]
if $n \ge n_0$ and $n_0$ is sufficiently large,
the threshold $n_0$ being independent of $d$, 
or for all $n\ge 1$ if $0 \le d \le 3$
(which would follow from the Hadamard conjecture).
The proofs depend on the probabilistic method,
and generalise previous
results that were restricted to the cases $d=0$ and $d=1$.
\end{abstract}

\pagebreak[4]

\section{Introduction}		\label{sec:intro}

Let $D(n)$ be the maximal determinant possible for an $n\times n$
matrix with elements in $\{\pm1\}$.
Hadamard~\cite{Hadamard}
\footnote{For
earlier contributions by Desplanques, 
L\'evy, 
Muir, 
Sylvester 
and Thomson (Lord Kelvin), 
see~\cite{Muir26,Sylvester} and~\cite[pg.~384]{MS}.}
proved that $D(n) \le n^{n/2}$, and the
\emph{Hadamard conjecture} is that a matrix achieving this upper bound
exists for each positive integer $n$ divisible by four.
The function $\Dbar(n) := D(n)/n^{n/2}$ 
is a measure of the sharpness of the Hadamard bound.
Clearly $\Dbar(n) = 1$ if a Hadamard matrix of order $n$
exists; otherwise $\Dbar(n) < 1$.  In this paper we
give lower bounds on $\Dbar(n)$.

Let $h$ be the maximal order of a Hadamard matrix subject to $h \le n$.
Then $d = n-h$ can be regarded as the ``gap'' between $n$ and the nearest
(lower) Hadamard order. We are interested the case that $n$ is not a 
Hadamard order, so (usually) $d > 0$ and $\Dbar(n) < 1$.

Except in the cases $d \in \{0,1\}$, previous lower bounds on $\Dbar(n)$
tended to zero as $n \to \infty$.
For example, the well-known bound of
Clements and Lindstr\"om~\cite[Corollary to Thm.~2]{CL65} shows that
$\Dbar(n) > (3/4)^{n/2}$, and \hbox{\cite[Thm.~9]{rpb249}} shows that
$\Dbar(n) \ge n^{-\delta/2}$, where $\delta := |n-h|$ (in this result
$h > n$ is allowed, so it is possible that $\delta < d$).
In contrast, we show that, for fixed~$d$,
$\Dbar(n)$ is bounded below by a positive constant
$\kappa_d$. We also show that, for all sufficiently large $n$,
$\Dbar(n) \ge (\pi e/2)^{-d/2}$. We conjecture that the 
``sufficiently large'' condition can be omitted; this is certainly true if
$d \le 3$.

Our lower bound proofs use the probabilistic method
pioneered by Erd\H{o}s (see for example~\cite{AS,ES}). In many cases the
probabilistic method gives sharper bounds than have been obtained by
deterministic methods. 
The probabilistic method does
not appear to have been applied previously to the Hadamard maximal
determinant problem, except (implicitly) in the case $d = 1$ (so $n \equiv 1
\bmod 4$); in this case the concept of \emph{excess} has been used~\cite{FK}, 
and lower bounds on the maximal excess were obtained by the probabilistic
method~\cite{Best,ES,FK}.  
In a sense our results generalise this idea, although we do not
directly generalise the concept of {excess} to cover $d > 1$.

Specifically, in our probabilistic construction
we adjoin $d$ extra columns to an $h \times h$ Hadamard matrix $A$,
and fill their $h \times d$ entries with random signs
obtained by independently tossing fair coins.
Then we adjoin $d$ extra rows, and fill their $d \times (h+d)$ entries with
$\pm 1$ values chosen deterministically 
in a way intended to approximately 
maximize the determinant of the final matrix $\widetilde{A}$.
To do so, we use the fact 
that this 
determinant can be expressed in terms 
of the  $d \times d$ Schur complement $(\widetilde{A}/A)$ of
$A$ in $\widetilde{A}$ (see~\S\ref{sec:Schur}).

In the case $d=1$, this method is essentially the same as the known method
involving the {excess} of the Hadamard matrix, and leads to the same
bounds that can be obtained by bounding the excess in a probabilistic
manner, as in~\cite{Best,BS,FK}. In this sense our method is
a generalisation of methods based on excess.

The structure of the paper is as follows:

\smallskip\noindent
\S\ref{sec:intro}:		Introduction\\
\S\ref{sec:notation}:		Notation\\
\S\ref{sec:Schur}:		The Schur complement lemma\\
\S\ref{sec:binomial}:		Some binomial sums\\
\S\ref{sec:construction}:	The probabilistic construction\\
\S\ref{sec:gaps}:		Gaps between Hadamard orders\\
\S\ref{sec:lemmas}:		Preliminary results\\
\S\ref{sec:bounds}:		Probabilistic lower bounds\\
\S\ref{sec:numerical}:		Numerical examples\\

Most of these section headings are self-explanatory.
\S\ref{sec:construction} describes the probabilistic construction
which is common to all our lower bound results.
\S\ref{sec:gaps} summarises some known results on gaps between Hadamard
orders. These results are relevant for bounding $d$ as a function of~$n$.

The main lower-bound results for $\Dbar(n)$,
which we now outline, are given in \S\ref{sec:bounds}.

Theorem~\ref{thm:small_d} obtains a lower bound on the expected
value of the determinant in a direct manner, by simply expanding the
determinant of the Schur complement as a sum of products. The difficulty
with this approach is that we have to consider $d!$ terms.
The ``diagonal'' term is expected to be larger than the other terms, but
in general
only by a factor of order $h$, so to obtain good bounds we need
$h$ of order at least $d!$.
Thus, this approach is only useful for small $d$.
Of course, the Hadamard conjecture implies 
that $d\le 3$.  However, what can currently
be \emph{proved} about gaps between Hadamard
orders is much weaker than this (see \S\ref{sec:gaps}).

For $d \le 3$, Theorem~\ref{thm:small_d}
shows that $\Dbar(n) \ge (\pi e/2)^{-d/2} > 1/9$, coming close to
Rokicki \emph{et al}'s conjectured lower bound of $1/2$
(see~\cite{Rokicki}),
and improving on earlier results~\cite{rpb249,CL65,Cohn63,KMS00,LL}
that failed to obtain a constant lower bound on $\Dbar(n)$ for 
$d \in \{2,3\}$.

Theorems~\ref{thm:lower_bd_via_Chebyshev}--\ref{thm:two_parameter_bound}
give slightly weaker bounds than Theorem~\ref{thm:small_d}, but
under less restrictive conditions on $d$ and $h$. For example,
Theorem~\ref{thm:lower_bd_via_Chebyshev} gives a nontrivial lower bound
whenever $h > \pi d^4/2$. By the results of Livinskyi~\cite{Livinskyi}
described in~\S\ref{sec:gaps}, this
condition holds for all sufficiently large~$n$.
Theorems~\hbox{\ref{thm:Chebyshev_Lovasz}--\ref{thm:two_parameter_bound}} 
further
weaken the conditions on $d$ and $h$. For example,
Theorem~\ref{thm:Cantelli_Hoeffding_Lovasz} is always applicable
if $h \ge 664$ and $d \ge 2$ (see Remark~\ref{remark:one_or_the_other}). For
$n <  668$ the Hadamard conjecture holds~\cite{KR}, so $d \le 3$ and
Theorem~\ref{thm:small_d} applies.
Thus, at least one of Theorem~\ref{thm:small_d}
or Theorem~\ref{thm:Cantelli_Hoeffding_Lovasz} always gives a nontrivial
lower bound on $\Dbar(n)$; this lower bound
is of order $(\pi e/2)^{-d/2}$.

To prove
Theorems~\ref{thm:lower_bd_via_Chebyshev}--\ref{thm:two_parameter_bound}
we need lower bounds on the determinant of a diagonally dominant
matrix.  Such bounds are provided by
Lemmas~\ref{lemma:super_duper_perturbation_bound}--
\ref{lemma:optimal_pert_bound} in~\S\ref{subsec:perturbation}.
The proofs of
Theorems~\ref{thm:lower_bd_via_Chebyshev}--\ref{thm:two_parameter_bound}
also require an upper bound on the variance of the diagonal elements
occurring in our probabilistic construction.  This is provided by
Lemma~\ref{lemma:variance}, which gives an exact formula for
the variance.

Other ingredients in the proofs of
Theorems~\ref{thm:lower_bd_via_Chebyshev}--\ref{thm:two_parameter_bound}
are the ``Lov\'asz Local Lemma'' of Erd\H{o}s and Lov\'asz~\cite{EL}
(for the proofs of 
Theorems~\ref{thm:Chebyshev_Lovasz}--\ref{thm:Cantelli_Hoeffding_Lovasz}),
and the well-known inequalities of
Hoeffding~\cite{Hoeffding} (for
Theorems~\ref{thm:Cantelli_Hoeffding_Lovasz}--\ref{thm:two_parameter_bound}),
Chebyshev~\cite{Chebyshev} (for 
Theorems~\ref{thm:lower_bd_via_Chebyshev}--\ref{thm:Chebyshev_Lovasz})
and Cantelli~\cite{Cantelli} (for
Theorems~\ref{thm:Cantelli_Hoeffding_Lovasz}--\ref{thm:two_parameter_bound}).

Finally, Theorem~\ref{thm:almost_all_n} gives the result
that
$\Dbar(n) \ge (\pi e/2)^{-d/2}$ for all $d \ge 0$ and
$n \ge n_0$, where $n_0$ is independent of $d$.
This follows from (a corollary of) Theorem~\ref{thm:two_parameter_bound}
by using known results on gaps between Hadamard orders.
We conjecture that the condition $n \ge n_0$ is unnecessary, and
that the inequality holds for all positive~$n$. The conjecture could
be proved/disproved by a finite (albeit large) computation, since
we have an explicit upper bound on $n_0$.

Theorems~\ref{thm:lower_bd_via_Chebyshev}--\ref{thm:Cantelli_Hoeffding_Lovasz}
are not quite strong enough to imply Theorem~\ref{thm:almost_all_n}. This is
because
Theorems~\ref{thm:lower_bd_via_Chebyshev}--\ref{thm:Cantelli_Hoeffding_Lovasz}
all involve a multiplicative ``correction factor'' of the form
\hbox{$(1 - O_d(1/h^{1/2}))$} in the lower bound~--
for example, the bounds~\eqref{eq:fudge1}--\eqref{eq:fudge2}
involve a correction factor $(1 - O(d^2/h^{1/2}))$.
Theorem~\ref{thm:two_parameter_bound} improves the ``correction factor'' 
to $(1 - O(d^{5/3}/h^{2/3}))$, which is close enough to unity to imply
Theorem~\ref{thm:almost_all_n} (the critical point being that the
exponent of $h$ is now greater than $1/2$). The price that we pay for this
improvement is that Theorem~\ref{thm:two_parameter_bound} involves a
parameter ($\lambda$) which must be chosen in a (close to) optimal way to give a
correction factor of the desired form, whereas 
Theorems~\ref{thm:lower_bd_via_Chebyshev}--\ref{thm:Cantelli_Hoeffding_Lovasz}
are explicit and do not involve any free parameters.
\pagebreak[3]

The constant $\pi e/2$ occurring in the bound $(\pi e/2)^{-d/2}$
of Theorem~\ref{thm:almost_all_n} is unlikely to be
optimal.  {From} the upper bounds of Barba~\cite{Barba33},
Ehlich~\cite{Ehlich64a,Ehlich64b} and Wojtas~\cite{Wojtas64} for $d \le 3$,
it seems plausible
that the optimal constant is $e/2$ and that the factor $\pi$ in our results
is a consequence of using the probabilistic method, which in some sense
estimates the mean rather than the maximum of a certain set of determinants.

\comment{
So far, the result has only been probabilistic, but since the success
probability is positive, that tells us that a suitable configuration of coin
tosses must {\it deterministically exist} so that an $n \times n$ sign
matrix indeed must exist satisfying our lower bound.  (And indeed the
argument could be derandomized using the ``method of conditional
expectations'' which would show there is a polynomial-time algorithm to find
a suitable $n \times n$ matrix?)  WDS

NB: I believe this if $d=1$ but I'm not sure about it if $d > 1$ because
of our way of handling the off-diagonal elements in the Schur complement.
That's why I've commented out this paragraph and made the following
weaker statement. RPB
} 

It is an open question whether our probabilistic construction can be
derandomized to give deterministic polynomial-time algorithms to construct
matrices satisfying the lower bounds given in \S\ref{sec:bounds}.
However, in practice we have been able to construct such
matrices using randomized algorithms based on the probabilistic
construction. The main practical difficulty is in constructing a Hadamard
matrix of maximal order $h \le n$, since numerous constructions for
Hadamard matrices are scattered throughout the literature.

In the special case $d=1$ our arguments simplify,
because there is no need to consider a nontrivial Schur complement or to
deal with the contribution of the off-diagonal elements.
This case was already considered by
Brown and Spencer~\cite{BS}, Erd\H{o}s and Spencer~\cite[Ch.~15]{ES}, 
and (independently) by 
Best~\cite{Best}; see also~\cite[\S2.5]{AS}
and ~\cite[Problem A4]{Putnam74}.
The consequence for lower bounds on $\Dbar(n)$ when
$n \equiv 1 \bmod 4$ was exploited by
Farmakis and Kounias~\cite{FK}, 
and an improvement using $3$-normalized Hadamard
matrices was considered by Orrick and Solomon~\cite{OS07}.
However, $3$-normalization does not seem to be helpful in the context of our
probabilistic construction.

Some of the results of this paper first
appeared in the (unpublished) manuscript~\cite{rpb253}. However, at that
time we did not have a proof of equation~\eqref{eq:sigma2} in
Lemma~\ref{lemma:variance} below (which gives the variance of the diagonal
terms in our probabilistic construction),
so we had to avoid using the variance and instead 
use Lemma~$15.2$ of~\cite{ES} (Lemma~$12$ of~\cite{rpb253}), 
which generally gives weaker results with more complicated proofs.

\subsubsection*{Acknowledgements}
We thank Robert Craigen for informing us of the work of his student Ivan
Livinskyi, and Will Orrick for his comments and for providing a copy of the
unpublished report~\cite{Rokicki}. Dragomir {\DJ}okovi\'c and Ilias
Kotsireas shared their list of known small Hadamard orders, which was
useful for checking our numerical computations.
The Mathematical Sciences Institute (ANU) provided computational
support on the cluster {\tt orac}.
The first author was supported in part 
by Australian Research Council grant DP140101417.		
\pagebreak[3]

\section{Notation}				\label{sec:notation}

We use the usual ``$O$'' and ``$o$'' notations.
$f \ll g$ means the same as $f = O(g)$,
and $f \gg g$ means the same as $g \ll f$.  The notations
$f \asymp g$ or $f = \Theta(g)$ mean that both $f \ll g$ and $f \gg g$.
Finally, $f = O_\delta(g)$ means that $f = O(g)$ when a parameter
$\delta$ is fixed, 
but the implicit constant may depend on $\delta$.

The binomial coefficient $\binom{m}{k}$ is defined to be zero if $k < 0$
or $k > m$. Thus, we can often avoid specifying upper and lower limits
of sums explicitly.

As in \S\ref{sec:intro}, $D(n)$ is the maximum determinant function
and $\Dbar(n) := D(n)/n^{n/2}$ is its normalization.
The set of orders of all Hadamard matrices is denoted by 
${\Had}$. 
If $n$ is given, then
$h\in\Had$ is always chosen to be maximal subject to $h \le n$,
so $d := n-h$ is minimal.
The case $d=0$ is trivial because then the Hadamard bound
applies, so we assume $d>0$ if this makes the statement of
the results simpler. We assume
$n \ge h\ge 4$ to avoid small special cases~--
it is easy to check if the results also hold for $1 \le n \le 3$ and 
$h \in \{1,2\}$.
In the asymptotic results, we can assume
that $d \ll h$. In fact, it follows from~\eqref{eq:gamma_Livinskyi} below that
$d \ll h^{1/6}$. 

Constants are denoted by $c$,
$c_1$, $c_2$, $\alpha$, $\beta$, etc.
Unless otherwise specified, $\ve$ is an arbitrarily small positive constant,
and $c = \sqrt{2/\pi} \approx 0.7979$.

Matrices are denoted by capital letters $A$ etc, and their elements by the
corresponding lower-case letters, e.g.\ $a_{i,j}$
or simply $a_{ij}$ if the meaning is clear.
 
The probability of an event $S$
is denoted by $\Prob[S]$,
the expectation of a random variable $X$ is denoted by $\E[X]$,
and the variance of $X$ by $\V[X]$.

$\mu(h)$ and $\sigma(h)^2$ are respectively the mean and variance of the
``diagonal'' elements occurring in our probabilistic construction~-- for
precise definitions see~\S\ref{subsec:variance}. We write simply
$\mu$ and $\sigma^2$ if $h$ is clear from the context.

\section{The Schur complement lemma}	\label{sec:Schur}

Let 
\begin{equation}		\label{eq:block_matrix}
\widetilde{A} = \left[\begin{matrix} A & B\\ C & D\\
\end{matrix}\right]
\end{equation} be an $n \times n$ matrix written in 
block form, where $A$ is $h \times h$, and $n = h + d > h$.
The \emph{Schur complement $(\widetilde{A}/A)$}
of $A$ in $\widetilde{A}$ is the $d \times d$
matrix 
$D - CA^{-1}B$ (see for example~\cite{Cottle,Higham}).
It is relevant to our problem in view of the following well-known
lemma~\cite{Brualdi,Schur}.
\begin{lemma}[Schur complement]	\label{lemma:Schur}
If $\widetilde{A}$ is as in~\eqref{eq:block_matrix}
and $A$ is nonsingular, then
\[\det(\widetilde{A}) = \det(A)\det(D - CA^{-1}B).\]
\end{lemma}
\begin{proof}
Take determinants of each side in the identity
\[
\left[\begin{matrix} A & B\\ C & D\\ \end{matrix}\right] =
\left[\begin{matrix} I & 0\\ CA^{-1} & I\\ \end{matrix}\right]
\left[\begin{matrix} A & B\\ 0 & D-CA^{-1}B\\ \end{matrix}\right]\,.
\]
\end{proof}
In our application of Lemma~\ref{lemma:Schur}, $A$ is a Hadamard
matrix of order~$h$, so $\det(A) = h^{h/2}$ (without loss of generality we
can assume that the sign is positive).  Thus, to maximize
$\det(\widetilde{A})$ for given $A$,
we need to maximize $\det(D - CA^{-1}B)$.
We can\,not generally find the exact maximum, but we can find
lower bounds on the maximum by using the probabilistic method.
For example, the mean is always a lower bound on the maximum.

\pagebreak[3]
\section{Some binomial sums}		\label{sec:binomial}

Lemma~\ref{lemma:Best} is a binomial sum which has appeared
several times in the literature, e.g.~Alon and Spencer~\cite[\S2.5]{AS},
Best~\cite[proof of Theorem 3]{Best}, Brown and Spencer~\cite{BS},
Erd\H{o}s and Spencer~\cite[proof of Theorem 15.2]{ES}.
It was also a problem in the 
1974 Putnam competition~\cite[Problem A4]{Putnam74}.
Lemma~\ref{lemma:Best} can be used
to calculate the mean of the diagonal terms that arise
when the probabilistic method is used to give lower bounds for
the Hadamard maximal determinant problem, as in~\cite{rpb253}
and our Lemma~\ref{lemma:variance}.

Lemma~\ref{lemma:double_sum} gives a closed-form expression for a
double sum which is analogous to the single sum of Lemma~\ref{lemma:Best}.
Lemma~\ref{lemma:double_sum} can be used to calculate the second
moments of the diagonal terms that arise when inequalities such as
Chebyshev's inequality are used 
to give lower bounds for the 
maximal determinant problem. 
In~\cite[Theorems 2--3]{rpb253} we gave lower bounds without using the
second moment, but these results can be improved (and the proofs simplified) 
by using estimates of the second moment.

For proofs of Lemmas~\ref{lemma:Best}--\ref{lemma:double_sum}
see~\cite{rpb255}. Generalisations are given in~\cite{BOOP}.

\begin{lemma}[Best \emph{et al}]	\label{lemma:Best}
For all $k \ge 0$,
\[
\sum_p \binom{2k}{k+p}\,|p| = k\binom{2k}{k}\,. 
\]
\end{lemma}

\begin{lemma}[Brent and Osborn]		\label{lemma:double_sum}
For all $k \ge 0$, 
\[
\sum_p \sum_q \binom{2k}{k+p} \binom{2k}{k+q}\, |p^2-q^2| =
 2k^2{\binom{2k}{k}}^2\,. 
\]
\end{lemma}

\section{The probabilistic construction}	\label{sec:construction}

We now describe the probabilistic construction that is common to
the proofs of
Theorems~\ref{thm:small_d}--\ref{thm:two_parameter_bound},
and prove some properties of the construction.
Our construction is a generalisation of Best's, which
is the case $d=1$.

Let $A$ be a Hadamard matrix of order $h\ge 4$.  
We add a border of $d$ rows
and columns to give a larger matrix $\widetilde{A}$ of order $n$. The 
border is defined by matrices $B$, $C$ and $D$ as in \S\ref{sec:Schur}.
The matrices
$A$, $B$, and $C$ have entries in $\{\pm1\}$.
We allow the matrix $D$ to have entries in $\{0, \pm 1\}$, but
the zero entries can be replaced by $+1$ or $-1$ without decreasing
$|\det(\widetilde{A})|$, so any lower bounds that we obtain on
$\max(|\det(\widetilde{A})|)$ are valid lower bounds on maximal 
determinants of $n\times n$ $\{\pm1\}$-matrices. To prove this, we
observe that, by Lemma~\ref{lemma:Schur}, $\det(\widetilde{A})
= \det(A)\det(D - CA^{-1}B)$ is a linear function of each element
$d_{ij}$ of $D$ (considered separately), so we can choose any ordering
of off-diagonal elements, then successively change
each off-diagonal element $d_{ij}$ of $D$ from $0$ to $+1$ or $-1$ in
such a way that $|\det(\widetilde{A})|$ does not decrease.

In the proofs of 
Theorems~\ref{thm:small_d}--\ref{thm:two_parameter_bound}
we show that our choice of $B$, $C$ and $D$ gives a Schur complement
$D - CA^{-1}B$ that, with positive probability, 
has sufficiently large determinant.
In the proof of Theorem~\ref{thm:small_d} it is sufficient to
consider $\E[\det(D - CA^{-1}B)]$; in the proofs of
Theorems~\ref{thm:lower_bd_via_Chebyshev}--\ref{thm:two_parameter_bound}
the argument is slightly more sophisticated and uses Chebyshev's inequality
or Cantelli's inequality.

\subsection{Details of the construction}	\label{subsec:details}
Let $A$ be any Hadamard matrix of order $h$.
$B$ is allowed to range over the set 
of all $h\times d$ $\{\pm 1\}$-matrices, chosen uniformly
and independently from the $2^{hd}$ possibilities.
The $d \times h$ matrix 
$C = (c_{ij})$ is a function of $B$. We choose
\[c_{ij} = \sign (A^TB)_{ji}\,,\]
where  
\[\sign(x) := \begin{cases} +1 \text{ if } x \ge 0,\\
			   -1 \text{ if } x < 0.\\
	     \end{cases}
\]
The definition of $\sign(0)$ here is arbitrary and does not
affect the results.
To complete the construction,
we choose $D = -I$. As mentioned above, it is inconsequential that
$D$ is not a $\{\pm1\}$-matrix.

\subsection{Properties of the construction}
				\label{subsec:properties}

Define $F = CA^{-1}B$
and $G = -(\widetilde{A}/A) = F-D = F+I$.
(The minus sign in the definition of $G$ is chosen
for convenience in what follows.)
Note that, since $A$ is a Hadamard matrix, $A^T = hA^{-1}$, so
$hF = CA^TB$. 

The definition of $C$ ensures that there is no cancellation in the inner
products defining the diagonal entries of $hF = C\cdot (A^T B)$. Thus, we
expect the diagonal entries $f_{ii}$ of $F$ to be nonnegative and 
of order $h^{1/2}$,
but the off-diagonal entries $f_{ij}$ ($i\ne j$) 
to be of order unity with high probability.
This intuition is justified by Lemma~\ref{lemma:sigma_asymptotics}.

The following lemma is (in the case $i=j$) due to M.~R.~Best~\cite{Best}
and independently J.~H.~Lindsey (see~\cite[footnote on 
pg.~88]{ES}). 
The upper bound can
be achieved infinitely often, in fact
whenever a regular Hadamard matrix of order $h$ exists.
For example, this is true if $h = 4q^2$, 
where $q$ is an odd prime power
and $q \not\equiv 7 \pmod 8$, see~\cite{XXS}.
\begin{lemma}			\label{lem:maxf}
If $F = (f_{ij})$ is chosen as above, then $|f_{ij}| \le h^{1/2}$
for $1 \le i, j \le d$.
\end{lemma}
\begin{proof}
This follows from the Cauchy-Schwarz inequality as in
Theorem~$1$ of Best~\cite{Best}. 
\end{proof}
\comment{Also,
J.-M. Goethals and J. J. Seidel [Canad. J. Math. 22 (1970), 597–614;
MR0282872 (44 \#106)] showed that ``if there is an Hadamard matrix of
order $n$ then there exists a regular symmetric Hadamard matrix with
constant diagonal of order $n^2$.''
However, the construction does not always give a maxdet matrix,
e.g. taking $h=36$ it gives the ``old record'' of $63×9^{17}×2^{36}$
and not the optimal value of $72×9^{17}×2^{36}$ which can be obtained
by the Orrick-Solomon $3$-normalised construction.
Similarly for $h=100$.
} 

\begin{lemma}		\label{lem:Efij}
If $F=(f_{ij})$ is chosen as above, then
\[
\E[f_{ij}] = \begin{cases} 2^{-h}h\binom{h}{h/2}\; \text{ if }\; i = j,\\
			     0 \;\text{ if }\; i \ne j.\\
             \end{cases}
\]
\end{lemma}
\begin{proof}
The case $i=j$ follows from Best~\cite[Theorem~3]{Best}.
The case $i \ne j$ is easy, since $B$ is chosen randomly.
\end{proof}

\begin{lemma}	\label{lemma:u_sum}
Let $A \in \{\pm1\}^{h\times h}$ be a Hadamard matrix,
$C \in \{\pm1\}^{d\times h}$, and $U = CA^{-1}$.  
Then, for each $i$ with $1 \le i \le d$,
\[\sum_{j=1}^h u_{ij}^2 = 1.\]
\end{lemma}
\begin{proof}
Since $A$ is Hadamard, $A^TA = hI$. Thus
$UU^T = 
h^{-1}CC^T$.
Since $c_{ij} = \pm 1$, ${\rm diag}(CC^T) = hI$.
Thus ${\rm diag}(UU^T) = I$.
\end{proof}

\begin{lemma}		\label{lem:fij2}
If $F = (f_{ij})$ is chosen as above, then
\begin{equation} 	\label{eq:Efij2}
\E[f_{ij}^2] = 1 \text{ for } \; i \ne j.
\end{equation}
\end{lemma}
\begin{proof} 
We can  assume, without essential loss of generality, that 
$i=1$, $j > 1$.  
Write $F = UB$, where $U = CA^{-1} = h^{-1}CA^T$.
Now 
\begin{equation}		\label{eq:sum_indep_vars}
f_{1j} = \sum_k u_{1k}b_{kj},
\end{equation}
where
\[u_{1k} = \frac{1}{h}\sum_{\ell} c_{1\ell}a_{k\ell}\]
and
\[c_{1\ell} = {\rm sgn}\left(\sum_m b_{m1}a_{m\ell}\right).\]
Observe that $c_{1\ell}$ and $u_{1k}$ depend only on the first column of $B$.
Thus, $f_{1j}$ depends only on the first and $j$-th columns of $B$.
If we fix the first column of $B$ and take expectations over all choices
of the other columns, we obtain
\[\E[f_{1j}^2] = \E\left[\sum_k\sum_{\ell}u_{1k}u_{1\ell}b_{kj}b_{\ell j}
  \right].\]
The expectation of the terms with $k\ne \ell$ vanishes,
and the expectation of the terms with $k=\ell$
is $\sum_k u_{1k}^2$.
Thus, \eqref{eq:Efij2} follows from Lemma~\ref{lemma:u_sum}.
\end{proof}

\begin{lemma}		\label{lem:independence}
Let $F = CA^{-1}B$ be chosen as above. Then
$f_{ij}$ and $f_{k\ell}$ are independent if and only if
$\{i,j\}\cap\{k,\ell\} = \emptyset$.
\end{lemma}
\begin{proof}
This follows from the fact that
$f_{ij}$ depends on columns $i$ and $j$ (and no other columns) of
$B$.
\end{proof}
Suppose that $i \ne j$, $k \ne \ell$. 
We cannot 
assume that $f_{ij}$
and $f_{k\ell}$ are independent. For example,
by Lemma~\ref{lem:independence},
$f_{12}$ and $f_{21}$ are
not independent. 
The following lemma bounds $\E[f_{ij}f_{k\ell}]$
without assuming independence.
\begin{lemma}		\label{lem:fijfkl}
Suppose that $i \ne j$, $k \ne \ell$. Then
\begin{equation}	\label{eq:E_fijfkl}
|\E[f_{ij}f_{k\ell}]| \le \E[|f_{ij}f_{k\ell}|] \le 1.
\end{equation}
\end{lemma}
\begin{proof}
The first inequality in~\eqref{eq:E_fijfkl} is immediate.
The second inequality follows from
the Cauchy-Schwarz inequality and Lemma~\ref{lem:fij2}, using
\[
\E[|f_{ij}f_{k\ell}|] \le \sqrt{\E[f_{ij}^2]\E[f_{k\ell}^2]} = 1.
\]
\end{proof}
\begin{lemma}		\label{lem:G_diag}
Let $G = F+I$ be chosen as above. Then
\[
\E\left[\prod_{i=1}^d g_{ii}\right] =
 \left[1 + 2^{-h}h\binom{h}{h/2}\right]^d.
\] 
\end{lemma}
\begin{proof}
By Lemma~\ref{lem:independence}, 
the diagonal terms $f_{ii}$ are independent;
hence the \hbox{diagonal} terms $g_{ii} = f_{ii}+1$ are independent.
Now $\E[g_{ii}] = \E[f_{ii}] +1$, so from 
Lemma~\ref{lem:Efij} we have
\[
\E\left[\prod_{i=1}^d g_{ii}\right] = \prod_{i=1}^d \E[g_{ii}] =
  \left[1 + 2^{-h}h\binom{h}{h/2}\right]^d.
\]
\end{proof}

\pagebreak[3]

\subsection{Mean and variance of elements of $G$}	\label{subsec:variance}

Using 
Lemma~\ref{lemma:double_sum}, we can
complete the computation of the mean and variance of the elements of the
matrix $G$.

\begin{lemma}	\label{lemma:variance}
Let $A$ be a Hadamard matrix of order $h \ge 4$ and $B$, $C$ be
$\{\pm1\}$-matrices chosen as above.  Let 
$G = CA^{-1}B + I$.
Then 
\begin{eqnarray}
\E[g_{ii}] &=& 1 + \frac{h}{2^h}\binom{h}{h/2}\,,\label{eq:mu1}\\
\E[g_{ij}] &=& 0 \text{ for } 1 \le i, j \le d,\; i \ne j,\label{eq:mu2}\\
\V[g_{ii}] &=&
  1 + \frac{h(h-1)}{2^{h+1}}\binom{h/2}{h/4}^2 
   - \; \frac{h^2}{2^{2h}}\binom{h}{h/2}^2,\label{eq:sigma2}\\
\V[g_{ij}] &=& 1 \text{ for } 1 \le i, j \le d,\; i \ne j.\label{eq:sigmaij}
\end{eqnarray}
\end{lemma}
\begin{proof}
Since $G = F+I$,
the results~\eqref{eq:mu1}, \eqref{eq:mu2} and \eqref{eq:sigmaij}
follow from
Lemma~\ref{lem:Efij} and Lemma~\ref{lem:fij2} above.
Thus, we only need to prove~\eqref{eq:sigma2}.
Since $g_{ii} = f_{ii}+1$, 
it is sufficient to compute $\V[f_{ii}]$.

Now $hF = CA^{T}B$ (since $A$ is a Hadamard matrix).  We compute the
second moment (about the origin) of the diagonal elements $hf_{ii}$
of $hF$.  Since $h$ is a Hadamard order and $h \ge 4$, we can write
$h = 4k$ where $k \in \Z$.  
Consider $h$ independent random variables $X_j \in \{\pm1\}$, $1 \le j \le h$,
where $X_j = +1$ with probability $1/2$.  Define random variables $S_1$,
$S_2$ by
\[S_1 = \sum_{j=1}^{4k} X_j\]
and
\[S_2 = \sum_{j=1}^{2k} X_j \; - \sum_{j=2k+1}^{4k} X_j\,.\]

Consider a particular choice of $X_1, \ldots, X_h$ and suppose that
$k+p$ of $X_1, \ldots, X_{2k}$ are $+1$,
and that $k+q$ of $X_{2k+1},\ldots,X_{4k}$ are $+1$.
Then we have $S_1 = 2(p+q)$ and $S_2 = 2(p-q)$.
Thus, taking expectations over all $2^{4k}$ possible (equally likely) choices
and using Lemma~\ref{lemma:double_sum}, we see that
\begin{eqnarray*}
\E[|S_1 S_2|] &=& 4\E[|p^2-q^2|]\\
 &=& \frac{4}{2^{4k}}\sum_p \sum_q \binom{2k}{k+p}\binom{2k}{k+q}|p^2-q^2|\\
 &=& \frac{4}{2^{4k}} \cdot 2k^2\binom{2k}{k}^2\\
 &=& \frac{h^2}{2^{h+1}}\binom{2k}{k}^2\,.\\
\end{eqnarray*}
By the definitions of $B$, $C$ and $F$, we see that
$hf_{ii}$ is a sum of the form $Y_1 + Y_2 + \cdots + Y_h$,
where each $Y_j$ is a random variable with the same distribution as $|S_1|$,
and each product $Y_j Y_\ell$ (for $j \ne \ell$) has the 
same distribution as $|S_1 S_2|$.
Also, $Y_j^2$ has the
same distribution as $|S_1|^2 = S_1^2$.  
The random variables $Y_j$ are not independent,
but by linearity of expectations we obtain
\[h^2\E[f_{ii}^2] = h\E[S_1^2] + h(h-1)\E[|S_1S_2|]
 = h^2 + h(h-1)\cdot \frac{h^2}{2^{h+1}}\binom{2k}{k}^2\,.\]
This gives
\[\E[f_{ii}^2] = 1 + \frac{h(h-1)}{2^{h+1}}\binom{2k}{k}^2\,.\]
The result for $\V[g_{ii}]$ now follows from
\[\V[g_{ii}] = \V[f_{ii}] = \E[f_{ii}^2] - \E[f_{ii}]^2\,.\]
\end{proof}

For convenience we write $\mu = \mu(h) := \E[g_{ii}]$ 
and $\sigma^2 = \sigma(h)^2 := \V[g_{ii}]$.
If $h$ is understood from the context we may write simply $\mu$
and $\sigma^2$ respectively.

We now give some asymptotic approximations to $\mu(h)$ and $\sigma(h)^2$ 
that are accurate for large~$h$.  We also show that $\mu(h)$ is monotonic
increasing and of order $h^{1/2}$, but $\sigma(h)$ is bounded and monotonic
decreasing.
\pagebreak[4]

\begin{lemma}	\label{lemma:sigma_asymptotics}
For $h \in 4\Z$, $h \ge 4$, 
$\mu(h)$ is monotonic increasing, and $\sigma(h)^2$ is monotonic decreasing.
Moreover, the following inequalities hold:
\begin{equation}	\label{eq:ineq_mu1}
\sqrt{\frac{2h}{\pi}} + 0.9 < \mu(h) < \sqrt{\frac{2h}{\pi}} + 1,
\end{equation}
and
\begin{equation}	\label{eq:ineq_mu2}
\mu(h) = 1 + \sqrt{\frac{2h}{\pi}}\left(1 - \frac{1}{4h} + 
		\frac{\alpha(h)}{h^2}\right)\raisecomma
\end{equation}
where 
\begin{equation}	\label{eq:alpha_bd}
0 \le \alpha(h) \le (4\sqrt{\pi} - 7)/2 < 0.04491.
\end{equation}
Also,
\begin{equation}	\label{eq:ineq_sigma1}
0.04507 \approx 1 - {3}/{\pi} = \lim_{h \to \infty} \sigma(h)^2
	< \sigma(h)^2 \le \sigma(4)^2 = 0.25,
\end{equation}
and
\begin{equation}	\label{eq:ineq_sigma2}
\sigma(h)^2 = \left(1 - \frac{3}{\pi}\right) + 
	\frac{11}{4\pi h} - \frac{\beta(h)}{h^2}\,\raisecomma
\end{equation}
where
\begin{equation}	\label{eq:beta_bd}
0 \le \beta(h) \le 12 - 37/\pi < 0.23.
\end{equation}
\end{lemma}
\begin{proof}
{From} the well-known asymptotic expansion of $\ln\Gamma(z)$ we obtain,
as in~\cite{KS}, 
an asymptotic expansion for the logarithm of the central binomial coefficient:
\begin{equation}		\label{eq:ln_central_binom}
\ln\binom{2m}{m} \sim m\ln 4 - \frac{\ln(\pi m)}{2}
	- \sum_{k \ge 1} \frac{B_{2k}(1-4^{-k})}{k(2k-1)}\,m^{1-2k}\,.
\end{equation}
Here the $B_{2k}$ are Bernoulli numbers, and $(-1)^{k+1}B_{2k}$ is
positive.
The sum is not convergent, but the terms in the sum alternate in sign
and the asymptotic expansion is
strictly enveloping in the sense of P\'olya and Szeg\"o\footnote{See,
for example, Corollaries~$2$--$3$ of arXiv:1608.04834v2.},
so upper and lower bounds may be found
by truncating the series after an even or an odd number of terms.

The inequalities \eqref{eq:ineq_mu1}--\eqref{eq:beta_bd}
now follow from a straightforward but tedious computation,
using the expressions for $\mu(h)$ and
$\sigma(h)^2$ in Lemma~\ref{lemma:variance}
and approximations obtained from~\eqref{eq:ln_central_binom} 
with $m=h/2$ and $m=h/4$.
Note that the leading terms (of order $h$) cancel in the computation
of $\sigma(h)^2$.

The monotonicity of $\mu(h)$ and $\sigma(h)^2$ follows from the inequalities
\eqref{eq:ineq_mu2}--\eqref{eq:alpha_bd} 
and \eqref{eq:ineq_sigma2}--\eqref{eq:beta_bd} respectively.
For example, from \eqref{eq:ineq_sigma2}, using the bounds on
$\beta(h)$ in~\eqref{eq:beta_bd}, we have
\[\sigma(h+4)^2 \le 1 - \frac{3}{\pi} + \frac{11}{4\pi(h+4)}
 <  1 - \frac{3}{\pi} + \frac{11}{4\pi h} - \frac{0.23}{h^2}
 \le \sigma(h)^2.\]
\end{proof}

\begin{remark}
{\rm
Because $\mu(h)$ is of order $h^{1/2}$ but $\sigma(h)^2$ is of order $1$,
the distribution of $g_{ii}$ is concentrated around the mean, and we expect
values smaller than $(1-\ve)\mu(h)$ to occur with low probability.
For fixed positive $\ve$, the probability should tend to zero as 
$h \to \infty$.
We can use Chebyshev's or Cantelli's inequality
to obtain bounds on this probability.
}
\end{remark}

\section{Gaps between Hadamard orders}	\label{sec:gaps}

In order to apply our results to obtain a lower bound on $\Dbar(n)$ for
given $n$, we need to know the order $h$ of a Hadamard matrix with
$h \le n$ and $n-h$ 
as small as possible.  Thus, it is necessary to consider the size of
possible gaps in the sequence
$(h_i)_{i \ge 1}$ of Hadamard
orders.  
We define the \emph{Hadamard gap function} $\gamma:\R \to \Z$ by
\begin{equation}	\label{eq:gamma_defn}
\gamma(x) := \max \{h_{i+1}-h_i \,|\, h_{i} \le x\} \cup \{0\}\,.
\end{equation}
In~\cite{rpb249,LL} it was shown,
using the Paley and Sylvester constructions, 
that $\gamma(n)$ can be bounded using the prime-gap function.
For example, if $p$ is an odd prime, then $2(p+1)$ is a Hadamard order.
However, only rather weak bounds on the prime-gap function are known.
A different approach which produces asymptotically-stronger 
bounds employs results of 
Seberry~\cite{Seberry}, as subsequently sharpened
by Craigen~\cite{Craigen1} and Livinskyi~\cite{Livinskyi}.
These results take the following form: 
for any odd positive
integer~$q$, a Hadamard matrix of order $2^t q$ exists for every integer
\[t \ge \alpha \log_2(q) + \beta,\]
where $\alpha$ and $\beta$ are 
constants.
Seberry~\cite{Seberry} obtained $\alpha=2$. 
Craigen~\cite{Craigen1}
improved this to $\alpha = 2/3$, $\beta = 16/3$,
and later obtained 
$\alpha = 3/8$ 
in unpublished 
work with Tiessen quoted
in~\cite[Thm.~2.27]{Horadam}
and \cite{Craigen2,CK}.%
\footnote{There are typographical errors
in~\cite[Thm.~2.27]{Horadam} and
in~\cite[Thm.~1.43]{CK}, 
where the floor function should
be replaced by the ceiling function. This has the effect of increasing
the additive constant $\beta$.}
Livinskyi~\cite{Livinskyi} found $\alpha = 1/5$, $\beta = 64/5$.
The results of Craigen and of Livinskyi
depend on the construction of Hadamard
matrices via signed groups, Golay numbers
and Turyn-type sequences~\cite{BDKR,KK,SY,Turyn74}.

The connection between these results and the
Hadamard gap function is given by Lemma~\ref{lemma:conversion}.  
{From} the lemma
and the results of Livinskyi, the Hadamard gap function satisfies
\begin{equation}	\label{eq:gamma_Livinskyi}
\gamma(n) = O(n^{1/6}).
\end{equation}
This is much sharper than  $\gamma(n) = O(n^{21/40})$ 
arising from the best current result for prime gaps (by
Baker, Harman and Pintz~\cite{BHP}),
although not as sharp as the result
$\gamma(n) = O(\log^2 n)$ that would follow from
Cram\'er's prime-gap conjecture~\cite{rpb249,Cramer,Shanks64,Silva}.
\begin{lemma}	\label{lemma:conversion}
Suppose there exist positive constants $\alpha$, $\beta$ such that
$2^t q \in \Had$
for all odd positive integers $q$
and all integers $t \ge \alpha \log_2(q) + \beta$.
Then the Hadamard gap function $\gamma(n)$ satisfies
\[\gamma(n) = O_{\beta}(n^{\alpha/(1+\alpha)})\,.
\]
\end{lemma}
\begin{proof}
Consider consecutive odd integers $q_0$, $q_1 = q_0+2$ and corresponding
$h_i = 2^tq_i$, where $t = \lceil\alpha \log_2(q_1) + \beta\rceil$. 
By assumption
there exist Hadamard matrices of orders $h_0$, $h_1$.
Also, $2^\beta q_1^\alpha \le 2^t < 2^{\beta+1}q_1^\alpha$.
Thus \[h_1 = 2^tq_1 \ge 2^\beta q_1^{1+\alpha}\]
and
\[h_1-h_0 
 = 2^{t+1} < 2^{\beta+2}q_1^\alpha
 \le 2^{\beta+2}(h_1 / 2^\beta)^{\alpha/(1+\alpha)}
 \le 4 \cdot 2^{\beta/(1+\alpha)}h_1^{\alpha/(1+\alpha)}.
\]
Now $2^t \le h_0$, so $h_1 = h_0 + 2^{t+1} \le 3h_0$.
Also, $\frac{\alpha}{1+\alpha} < 1$ and $\frac{1}{1+\alpha} < 1$.
Thus
\[h_1-h_0 
  < 12\cdot 2^\beta h_0^{\alpha/(1+\alpha)}
  = O_{\beta}(h_0^{\alpha/(1+\alpha)}).\]
\end{proof}

\section{Preliminary results}			\label{sec:lemmas}

We now state some well-known results
(Propositions \ref{prop:Chebyshev}--\ref{prop:Lovasz})
and prove some lemmas
that are needed in \S\ref{sec:bounds}. 

\subsection{Probability inequalities}	\label{subsec:probability}

Proposition~\ref{prop:Chebyshev} is the well-known inequality
of Chebyshev~\cite{Chebyshev},
and Proposition~\ref{prop:Cantelli} is a one-sided
analogue due to Cantelli~\cite{Cantelli}.
\begin{proposition}[Chebyshev]	\label{prop:Chebyshev}
Let $X$ be a random variable with finite mean $\mu = \E[X]$ 
and finite variance
$\sigma^2 = \V[X]$.  Then, for all $\lambda > 0$,
\[
\Prob[\,|X-\mu| \ge \lambda\,] \le \frac{\sigma^2}{\lambda^2}\,\raisedot
\]
\end{proposition}

\begin{proposition}[Cantelli]	\label{prop:Cantelli}
Let $X$ be a random variable with finite mean $\mu = \E[X]$ 
and finite variance $\sigma^2 = \V[X]$. Then, for all $\lambda > 0$,
\[
\Prob[X-\mu \ge \lambda] \le \frac{\sigma^2}{\sigma^2+\lambda^2}
\;\;\text{ and }\;\;
\Prob[X-\mu \le -\lambda] \le \frac{\sigma^2}{\sigma^2+\lambda^2}\,\raisedot
\]
\end{proposition}

Proposition~\ref{prop:Hoeffding} is
a two-sided version of Hoeffding's
tail inequality. A~one-sided version is proved
in~\cite[Theorem 2]{Hoeffding}.
Hoeffding's inequality gives a sharper bound than Chebyshev's inequality
in the case that the random variable~$X$ is a sum of independent,
bounded random variables~$X_i$.
\begin{proposition}[Hoeffding] 	\label{prop:Hoeffding}
Let $X_1,\ldots,X_h$ be independent random variables
with sum
$X = X_1 + \cdots + X_h$. 
Assume that
$X_i \in [a_i, b_i]$ and, for some~$i\le h$, $a_i < b_i$.
Then, for all $\lambda > 0$,
\[
\Prob[\,|X - E[X]| \ge \lambda\,] \;\le\; 
 2\,\exp\left(\frac{-2\lambda^2}{\sum_{i=1}^h (b_i-a_i)^2} \right)\,\raisedot
\]
\end{proposition}

We also need the ``symmetric'' case of the Lov\'asz Local Lemma~\cite{EL}, 
where ``symmetric'' means that
the upper bound on the probability of each event is the same.
We state the formulation given in~\cite[Corollary 5.1.2]{AS}, with a slight
change of notation.

\begin{proposition}[Lov\'asz Local Lemma, symmetric case] \label{prop:Lovasz}
Let $E_1, E_2, \ldots E_m$ be events in an arbitrary probability space.
Suppose that each event $E_i$ is mutually independent of all the other
events $E_j$ except for at most $D$ of them, and that
$\Prob[E_i] \le p$ for $1 \le i \le m$.  If
\begin{equation}
ep(D+1) \le 1,		\label{eq:Lovasz-sufficient1}
\end{equation}
then
\[\Prob\left[\bigwedge_{i=1}^m \overline{E_i}\right] > 0\]
{\rm (in other words, with positive probability none of the events $E_i$
holds).}
\end{proposition}
\begin{remark}
{\rm
It follows from a theorem of Shearer~\cite{Shearer} that the
inequality~\eqref{eq:Lovasz-sufficient1} can be replaced by
$epD \le 1$.  This improvement would make little difference
to our results, so we use the better-known
condition~\eqref{eq:Lovasz-sufficient1}.
}
\end{remark}

\subsection{Perturbation bounds}	\label{subsec:perturbation}

We state some lower bounds on the determinant of a matrix
which is close to the identity matrix. 
Lemma~\ref{lemma:super_duper_perturbation_bound} generalises and
sharpens some inequalities due to
Ostrowski~\cite{Ostrowski37a,Ostrowski38}
and von Koch~\cite{Koch}.
Note that the condition on $e_{ii}$ in
Lemma~\ref{lemma:super_duper_perturbation_bound}
is one-sided.  This is useful if we want to apply
Cantelli's inequality, as in the proofs of
Theorems~\ref{thm:Cantelli_Hoeffding_Lovasz}--%
\ref{thm:two_parameter_bound} below.

\begin{lemma}	\label{lemma:super_duper_perturbation_bound}
If $M = I-E \in \R^{d\times d}$, where $|e_{ij}| \le \ve$
for $i \ne j$ and $e_{ii} \le \delta$ for $1 \le i \le d$, 
where $\delta\ge 0$ and $\delta + (d-1)\ve \le 1$, then
\[\det(M) \ge (1 - \delta - (d-1)\ve)
	      (1 - \delta + \ve)^{d-1}\,.\]
\end{lemma}
\begin{proof}
See~\cite[Corollary~1]{rpb258}.
\end{proof}

\begin{lemma}	\label{lemma:optimal_pert_bound}
If $M = I - E \in \R^{d\times d}$, 
$|e_{ij}| \le \ve$ for $1 \le i, j \le d$,
and $d\ve \le 1$, then 
\[\det(M) \ge 1 - d\ve.\]
\end{lemma}
\begin{proof}
This is implied by the case $\delta = \ve$ of
Lemma~\ref{lemma:super_duper_perturbation_bound}.
It is also implicit in a bound due to Ostrowski~\cite[eqn.~(5,5)]{Ostrowski38}.
\end{proof}

\subsection{An inequality involving $h$ and $n$}
	\label{subsec:inequalities}

Lemma~\ref{lemma:uncond2} 
allows us to deduce inequalities
involving~$n^n$ from corresponding inequalities involving~$h^n$.

\begin{lemma} \label{lemma:uncond2}
If $n = h+d > h > 0$, then
\[(h/n)^n > \exp(-d - d^2/h).\]
\end{lemma}
\begin{proof}
Writing $x = d/n$,
the inequality 
$\ln(1-x) > -x/(1-x)$
implies that
\[(1-x)^n \,>\, \exp\left(-\frac{nx}{1-x}\right)\,.\]
Since $1-x = h/n$, we obtain
\[\left(\frac{h}{n}\right)^n  
 \,>\, \exp\left(\frac{-d}{1-d/n}\right) 
  = \exp(-d - d^2/h).\]
\end{proof}

\pagebreak[3]
\section{Probabilistic lower bounds} \label{sec:bounds}

In this section we prove several lower bounds on $D(n)$ and $\Dbar(n)$
where, as usual, $n = h+d$ and $h$ is the order of a Hadamard matrix.
Theorem~\ref{thm:small_d} assumes that $d \le 3$;
Theorems~\ref{thm:lower_bd_via_Chebyshev}--\ref{thm:almost_all_n}
allow $d > 3$. Theorem~\ref{thm:small_d} can be extended to allow
$d>3$, but only on the assumption that $n$ is sufficiently large~--
see Theorem~$1$ of~\cite{rpb253}.

First we state a Lemma which is useful in its own right,
and is required for the proof of Theorem~\ref{thm:small_d}.

\begin{lemma}			\label{lemma:expand_det}
If $n = h+d$ where $4 \le h \in {\Had}$ and $1 \le d \le 3$, then
\[D(n) \ge h^{h/2}(\mu^d - \eta),\]
where $\mu = \mu(h)$ is as in~$\S\ref{subsec:variance}$, and
\[
\eta = \eta(h,d) =
\begin{cases}
	0 \text{ if } d = 1,\\
	1 \text{ if } d = 2,\\
	5h^{1/2} + 3 \text{ if } d = 3.\\
\end{cases}
\]
\end{lemma}
\begin{proof}
We use the probabilistic construction and notation of
\S\ref{sec:construction}.  
Let $A$ be a Hadamard matrix of order~$h$. Define matrices $B, C, D, F$ and
$G$ as in \S\ref{sec:construction}.

For notational convenience we give the proof for the case $d=3$.  The cases
$d \in \{1,2\}$ are similar (but easier).

Since $G = F + I$, we have $g_{ii} = f_{ii}+1$ and
\[G = \left[
\begin{array}{ccc}
g_{11} & f_{12} & f_{13}\\
f_{21} & g_{22} & f_{23}\\
f_{31} & f_{32} & g_{33}\\
\end{array}
\right]\,.
\]
Expanding $\deg(G)$ we obtain $d! = 6$ terms. The
``diagonal'' term is $g_{11}g_{22}g_{33}$. There are
$3$ terms involving one diagonal element, for example $-f_{12}f_{21}g_{33}$,
and $2$ terms involving no diagonal elements, for example
$f_{12}f_{23}f_{31}$. Define the \emph{type} of a term to be the number
of diagonal elements that it contains. Thus the diagonal term has type $3$
(or type $d$ in general). Let $T_k$ be an upper bound on the
magnitude of the expectation of
a term of type $k$.  Then
\begin{equation}	\label{eq:EdetG3}
E[\det(G)] \ge E[g_{11}g_{22}g_{33}] - 3T_1 - 2T_0.
\end{equation}
Now, by Lemmas~\ref{lem:maxf} and~\ref{lem:fijfkl},
\[|E[f_{12}f_{23}f_{31}]| \le E[|f_{12}f_{23}|]\cdot\max|f_{31}|
  \le 1\cdot h^{1/2} = h^{1/2},
\]
so we can take $T_0 = h^{1/2}$.  Similarly,
\[|E[f_{12}f_{21}g_{33}]| \le E[|f_{12}f_{23}|]\cdot\max|g_{31}|
\le h^{1/2}+1,
\]
so we can take $T_1 = h^{1/2}+1$.
Also, from Lemma~\ref{lem:G_diag} and the definition of $\mu(h)$, we have
\[E[g_{11}g_{22}g_{33}] = \mu^3.\]  Thus, from~\eqref{eq:EdetG3}, we obtain
\[E[\det(G)] \ge \mu^3 - 3(h^{1/2}+1) - 2h^{1/2} = \mu^3 - 5h^{1/2} - 3.\]
We have shown that, with $\eta = \eta(h,d)$ as in the statement of the Lemma,
\begin{equation}	\label{eq:eta_bound}
E[\det(G)] \ge \mu^d - \eta
\end{equation}
holds for $d = 3$. The proofs for $1 \le d \le 2$ are similar but simpler.

{From}~\eqref{eq:eta_bound},
there exists some assignment of signs to the elements of $B$
such that, for the resulting matrix $G$, we have
\begin{equation}	\label{eq:G_eta_bound}
\det(G) \ge \mu^d - \eta.
\end{equation}
Hence, by the Schur complement lemma (Lemma~\ref{lemma:Schur}),
\[D(n) \ge h^{h/2}\det(G) \ge h^{h/2}(\mu^d - \eta).\]
\end{proof}
\begin{remark}
{\rm
The restriction $d \le 3$ in Lemma~\ref{lemma:expand_det}
is not necessary.  In the general case,
a similar argument, given in~\cite[pg.~$13$]{rpb253}, 
shows that
\[D(n) \ge h^{h/2}(\mu(h)^d - \eta(h,d)),\]
where 
\[\eta(h,d) \le (d!-1)(h^{1/2}+1)^{d-2}\;\text{ for }\; d \ge 2.\]
It follows from Lemma~\ref{lemma:sigma_asymptotics} that
$\eta(h,d)/\mu(h)^d = O(d!(\pi/2)^{d/2}/h)$.
Because of the factor $d!(\pi/2)^{d/2}$ in this bound,
the general result is useless unless $d$ is small, say $d \le 3$ or
$d \le 4$.
Theorems~\ref{thm:lower_bd_via_Chebyshev}--\ref{thm:two_parameter_bound} 
overcome this difficulty by using
Lemma~\ref{lemma:super_duper_perturbation_bound} or 
Lemma~\ref{lemma:optimal_pert_bound},
avoiding the expansion of $\det(G)$ as a sum of $d!$ terms.
}
\end{remark}

We now deduce Theorem~\ref{thm:small_d} from Lemma~\ref{lemma:expand_det}.
For $d > 3$ a similar result holds, but we can 
only prove it for $n$ sufficiently large~--
see Theorem~\ref{thm:almost_all_n}
and also the weaker result of~\cite[Corollary~$2$]{rpb253}.

\begin{theorem}	\label{thm:small_d}
If $0\le d\le 3$, $h \in {\Had}$, $h \ge 4$, and $n = h+d$, then
\[
\Dbar(n) \ge \left(\frac{2}{\pi e}\right)^{d/2}.
\]
Moreover, the inequality is strict if $d > 0$.
\end{theorem}
\begin{proof}
Define $c := \sqrt{2/\pi}$ and $K := 0.9/c$.
The result is trivial if $d=0$, so assume that $1 \le d \le 3$.
Since $d$ is bounded, we can ignore functions of $d$ multiplying the
``$O$'' terms.
Lemma~\ref{lemma:expand_det} gives
\[D(n) \ge h^{h/2}(\mu^d - \eta),\] 
and from
Lemma~\ref{lemma:sigma_asymptotics} we have
$\mu \ge ch^{1/2} + 0.9$,
so $\mu^d \ge c^dh^{d/2}(1 + dKh^{-1/2})$. 
Thus
\[D(n) \ge c^d h^{n/2}\left(1 + dKh^{-1/2} - \frac{\eta}{c^dh^{d/2}}\right).\]
{From} Lemma~\ref{lemma:uncond2}, $(h/n)^{n} \ge \exp(-d - d^2/h)$, so
\[\Dbar(n) = \frac{D(n)}{n^{n/2}}
 \ge c^d e^{-d/2}\left(1 + dKh^{-1/2} - \frac{\eta}{c^dh^{d/2}}\right)
	e^{-d^2/(2h)}.\]
Since
$c^d e^{-d/2} = ({2}/{(\pi e)})^{d/2}$,
$K$ is positive, and $\eta/h^{d/2} = O(h^{-1})$,
the term $dKh^{-1/2}$ dominates the $O(h^{-1})$ terms, and
the result follows for all sufficiently large~$h$.
In fact, some computation shows that this
argument is sufficient for $d \in \{1,2\}$ and all $h \ge 4$.
For $d=3$ we obtain
\begin{eqnarray*}
\Dbar(n)
	 &\ge& \left(\frac{2}{\pi e}\right)^{3/2}
	\left(1 + \frac{3K}{h^{1/2}} - \frac{5h^{1/2}+3}{c^3 h^{3/2}}\right)
	e^{-4.5/h}.
\end{eqnarray*}
This shows that $\Dbar(n) \ge (2/(\pi e))^{3/2}$ for $h \ge 28$.
Thus, we only have to consider the cases
$n \in \{7,11,15,19,23,27\}$.  Now for $n = 4k-1$, where $4k \in \Had$,
an easy argument of Sharpe~\cite{Sharpe} involving minors of a 
Hadamard matrix of order $4k$, as in~\cite[Theorem~$2$]{KMS00}, shows that
\[D(4k-1) \ge \frac{D(4k)}{4k} = (4k)^{2k-1},\]
so
\[\Dbar(4k-1) \ge (4k)^{2k-1}/(4k-1)^{(4k-1)/2}.\]
This is sufficient to show that 
$\Dbar(n) > (2/(\pi e))^{3/2}$ for $n = 4k-1 \le 27$.
\end{proof}
\begin{corollary}	\label{cor:assume_H}
The Hadamard conjecture implies that $\Dbar(n)$ is bounded below
by a positive constant.
\end{corollary}
\begin{proof}
If the Hadamard conjecture is true, then for 
$4 < n \not\equiv 0 \pmod 4$, we can take
$h = 4\lfloor n/4\rfloor$ and $d = n-h \le 3$ in
Theorem~$\ref{thm:small_d}$. This gives
\[1 > \Dbar(n) > 
 \left(\frac{2}{\pi e}\right)^{d/2} \!\!\ge\;
 \left(\frac{2}{\pi e}\right)^{3/2} \!>\; 0.1133\,.\]
\end{proof}

\begin{remark}
{\rm 
It is interesting to compare our Theorem~$\ref{thm:small_d}$ 
(or the slightly sharper Lemma~\ref{lemma:expand_det}) with
Theorem~$2$ of Koukouvinos, Mitrouli and Seberry \cite{KMS00},
assuming the existence of the relevant Hadamard matrices.
In the case $n \equiv 2 \bmod 4$, the bound given
by our Theorem~$\ref{thm:small_d}$ 
(respectively Lemma~\ref{lemma:expand_det})
is better for $n \ge 22$ (resp.\ $14$) than the bound
$2(n+2)^{(n-2)/2}/n^{n/2} \sim 2e/n$ implied by~\cite[Theorem~$2$]{KMS00}.
In the case $n \equiv 3 \bmod 4$, the bound given by our
Theorem~$\ref{thm:small_d}$
(resp.\ Lemma~\ref{lemma:expand_det})
is better for $n \ge 211$ (resp.\ $135$)
than the bound
$(n+1)^{(n-1)/2}/n^{n/2} \sim (e/n)^{1/2}$ 
implied by~\cite[Theorem~$2$]{KMS00}.
}
\end{remark}

We now prove several theorems which apply for arbitrarily
large~$d$. The proofs depend on the fact that
$\sigma(h)$ is bounded (see Lemma~\ref{lemma:variance}).
This enables us to use Chebyshev's inequality (or Cantelli's inequality).

Theorems~\ref{thm:lower_bd_via_Chebyshev}--\ref{thm:two_parameter_bound}
give lower bounds on $\det(G)/\mu^d$; these are easily translated
into lower bounds on $D(n)$, since
$D(n) \ge h^{h/2}\det(G)$ (by the Schur complement lemma),
and $\mu > \sqrt{2h/\pi} + 0.9$ (by Lemma~\ref{lemma:sigma_asymptotics}).
Each of the
Theorems~\ref{thm:lower_bd_via_Chebyshev}--\ref{thm:two_parameter_bound}
is followed by a corollary which gives a
corresponding lower bound on $\Dbar(n)$.

\begin{theorem}		\label{thm:lower_bd_via_Chebyshev}
Suppose $d \ge 1$, $4 \le h \in {\cal H}$, $n = h+d$, $G$ as
in~$\S\ref{subsec:properties}$. Then, with positive probability 
\begin{equation}
\frac {\det G}{\mu^d} \ge 1 - \frac{d^2}{\mu}\,.	\label{eq:fudge1}
\end{equation}
\end{theorem}
\begin{proof}
Let $\lambda$ be a positive parameter to be chosen later, and $\mu = \mu(h)$.
For the purposes of this proof, 
we say that $G$ is \emph{good} if the conditions of
Lemma~\ref{lemma:optimal_pert_bound}
apply with $M = \mu^{-1}G$ and $\ve = \lambda/\mu$.
Otherwise $G$ is \emph{bad}.

Assume $1 \le i, j \le d$. 
{From} Lemma~\ref{lemma:variance}, $V[g_{ij}] = 1$ for $i \ne j$;
from Lemma~\ref{lemma:sigma_asymptotics},
$V[g_{ii}] = \sigma(h)^2 \le 1/4$.
It follows from Chebyshev's
inequality (Proposition~\ref{prop:Chebyshev}) that
\[\Prob[|g_{ij}| \ge \lambda] \le \frac{1}{\lambda^2}
\;\text{ for }\; i\ne j,\]
and
\[\Prob[|g_{ii} - \mu| \ge \lambda] \le
 \frac{\sigma^2}{\lambda^2}\,\raisedot\]
Thus,
\[\Prob[G \text{ is bad}] \le \frac{d(d-1)}{\lambda^2}
	+ \frac{d\sigma^2}{\lambda^2}
	= \frac{d(d+\sigma^2-1)}{\lambda^2}
                           < \frac{d^2}{\lambda^2}
\,\raisedot
\]

Taking $\lambda = d$ gives
$\Prob[G \text{ is bad}] < 1$, 
so $\Prob[G \text{ is good}]$ is positive.
Whenever $G$ is good we can apply Lemma~\ref{lemma:optimal_pert_bound}
to $\mu^{-1}G$, obtaining
$\det(\mu^{-1}G) \ge 1 - d\ve = 1 - d\lambda/\mu = 1 - d^2/\mu$.
\end{proof}
\begin{remark}		\label{remark:improved_simple}
{\rm
With the optimal choice $\lambda = \sqrt{d(d+\sigma^2-1)}$ we obtain the
less elegant but slightly sharper result that, with positive probability,
\[\frac{\det(G)}{\mu^d} \ge 
	1 - \frac{\sqrt{d^3(d+\sigma^2-1)}}{\mu}\,\raisedot\]
}
\end{remark}

\begin{corollary} \label{cor:lower_bd_via_Chebyshev}
Under the conditions of Theorem~$\ref{thm:lower_bd_via_Chebyshev}$,
\begin{equation}
{\cal R}(n) \ge \left(\frac{2}{\pi e}\right)^{d/2}
	\left(1 - d^2\sqrt{\frac{\pi}{2h}}\right)\,. \label{eq:fudge2}
\end{equation}
\end{corollary}
\begin{proof}
Write $c := \sqrt{2/\pi}$.  
We can assume that $d^2 < ch^{1/2}$, for there is nothing to prove unless
the right side of~\eqref{eq:fudge2} is positive.
{From} Lemma~\ref{lemma:sigma_asymptotics},
$ch^{1/2} < \mu$, so $d^2 < \mu$.

{From} Theorem~\ref{thm:lower_bd_via_Chebyshev} and the Schur complement
Lemma,
\[\Dbar(n) \ge \frac{h^{h/2} \mu^d}{n^{n/2}} \left(1 -
\frac{d^2}{\mu}\right)\,\raisedot\]
Using  $ch^{1/2} < \mu$, this gives
\[\Dbar(n) \ge c^d (h/n)^{n/2}(1 - d^2/\mu).\]
By Lemma~\ref{lemma:uncond2}, $(h/n)^n > \exp(-d-d^2/h)$, so
\begin{equation}	\label{eq:C1bd1}
\Dbar(n) \ge c^d e^{-d/2} f
 = \left(\frac{2}{\pi e}\right)^{d/2}f,
\end{equation}
where 
\begin{equation}	\label{eq:C1f}
f = \exp\left(-\,\frac{d^2}{2h}\right) \left(1 - \frac{d^2}{\mu}\right).
\end{equation}
Thus, to prove~\eqref{eq:fudge2}, it suffices to prove that
\[f \ge 1 - \frac{d^2}{ch^{1/2}}\,\raisedot\]
Since
\[\exp\left(-\,\frac{d^2}{2h}\right) \ge 1 - \frac{d^2}{2h}\,\raisecomma\]
it suffices to prove that
\begin{equation}	\label{eq:ineq_dhmu}
 \left(1 - \frac{d^2}{2h}\right)
 \left(1 - \frac{d^2}{\mu}\right)
 \ge 1 - \frac{d^2}{ch^{1/2}}\,\raisedot
\end{equation}
Expanding and simplifying shows that the inequality~\eqref{eq:ineq_dhmu}
is equivalent to
\begin{equation}	\label{eq:ineq_dhmu2}
2h + \mu \le d^2 + \mu\sqrt{2\pi h}.
\end{equation}
Now, by Lemma~\ref{lemma:sigma_asymptotics}, 
$\mu > c\sqrt{h} + 0.9$, so
$\mu\sqrt{2\pi h} > 2h + 0.9\sqrt{2\pi h}$
(using  $c\sqrt{2\pi} = 2$).
Thus, to prove~\eqref{eq:ineq_dhmu2}, it suffices to show that
$\mu \le d^2 + 0.9\sqrt{2\pi h}$. 
Using Lemma~\ref{lemma:sigma_asymptotics} again, we have
$\mu \le ch^{1/2} + 1$, so it suffices to show that
\[ch^{1/2} + 1 \le 0.9\sqrt{2\pi h} + d^2.\]
This follows from $c \le 0.9\sqrt{2\pi}$ and $1 \le d^2$,
so the proof is complete.
\end{proof}

\begin{remark}	\label{remark:lower_bd_via_Chebyshev}
{\rm
Corollary~\ref{cor:lower_bd_via_Chebyshev} gives a nontrivial lower bound
on ${\cal R}(n)$ iff the second factor in the bound is positive, 
i.e.~iff $h > {\pi}d^4/2$.
By Livinskyi's results~\cite{Livinskyi}, this condition 
holds for all sufficiently large $n$ 
(assuming as always that we choose the maximal $h$ for given $n$).
{From} Theorem~\ref{thm:small_d}, 
the second factor in~\eqref{eq:fudge2} can be
omitted if $d \le 3$.  
}
\end{remark}
We now improve on Theorem~\ref{thm:lower_bd_via_Chebyshev}, if 
$d$ is sufficiently large, by using 
the Lov\'asz Local Lemma~\cite{EL} (Proposition~\ref{prop:Lovasz}).

\begin{theorem}		\label{thm:Chebyshev_Lovasz}
Suppose $d \ge 1$, $4 \le h \in {\cal H}$, $n = h+d$, $G$ as
in~$\S\ref{subsec:properties}$.
Then with positive probability
\[\frac{\det G}{\mu^d} \ge 1 - \frac{2d\sqrt{(d-1)e}}{\mu}\,\raisedot\]
\end{theorem}
\begin{proof}
If $d=1$ the result is easy, since $\det G = g_{11} \ge \E[g_{11}] = \mu$
with positive probability.
Thus, we can assume that $d \ge 2$.

Let $\lambda$ be a positive parameter to be chosen later.
As in Theorem~\ref{thm:lower_bd_via_Chebyshev},
for the purposes of this proof
we say that $G$ is \emph{good} if the conditions of
Lemma~\ref{lemma:optimal_pert_bound}
apply with $M = \mu^{-1}G$ and $\ve = \lambda/\mu$.
Otherwise $G$ is \emph{bad}.

Let $E_{ij}$ be the event that 
$|g_{ij}| > \lambda$ (if $i\ne j$)
or $|g_{ii}-\mu| > \lambda$ (if $i=j$).
Thus $G$ is {good} if none of the $E_{ij}$ holds.

{From} Lemma~\ref{lemma:variance} (if $i \ne j$) and
Lemma~\ref{lemma:sigma_asymptotics} (if $i=j$), we have
$\V[g_{ij}] \le 1$ in both cases. Thus, from Chebyshev's inequality,
$\Prob[E_{ij}] \le \lambda^{-2}$.

Now, by Lemma~\ref{lem:independence},
$E_{ij}$ is independent of $E_{k\ell}$ if 
$\{i,j\} \cap \{k,\ell\} = \emptyset$.
Thus, in Proposition~\ref{prop:Lovasz} we can take
$D = 4d-5$, and the proposition shows that $G$
is good with positive probability
provided that $\lambda^2 \ge 4e(d-1)$.  
We take the smallest positive $\lambda$ satisfying this inequality,
i.e.~$\lambda = 2\sqrt{e(d-1)}$.  Now the result follows from
the inequality
\begin{equation}
 \frac{\det G}{\mu^d} \ge 1 - \frac{d\lambda}{\mu}\,\raisecomma
	\label{eq:1dlammu}
\end{equation}
which holds whenever $G$ is good, by Lemma~\ref{lemma:optimal_pert_bound} 
applied to $M = \mu^{-1}G$. 
\end{proof}

\begin{corollary}	\label{cor:Chebyshev_Lovasz}
Under the conditions of Theorem~$\ref{thm:Chebyshev_Lovasz}$,
\[
{\cal R}(n) \ge \left(\frac{2}{\pi e}\right)^{d/2}
		\left(1 - d\sqrt{\frac{2\pi e(d-1)}{h}}\right)
	e^{-d^2/(2h)}\,.
\]
\end{corollary}
\begin{proof}[Proof (sketch).]
This is similar to the proof of~\eqref{eq:C1bd1}--\eqref{eq:C1f} above,
using the bound of Theorem~\ref{thm:Chebyshev_Lovasz}
instead of the bound of Theorem~\ref{thm:lower_bd_via_Chebyshev}.
\end{proof}

\begin{remark}
{\rm
In Corollary~\ref{cor:lower_bd_via_Chebyshev} we absorbed
the factor $e^{-d^2/(2h)}$ into the final bound.  We do not attempt
to do this in Corollary~\ref{cor:Chebyshev_Lovasz} because the exponent
of $d$ in the ``main term'' is~$3/2$ rather than~$2$.  However, 
by~\eqref{eq:gamma_Livinskyi} above, $d \ll h^{1/6}$, so
$(d^2/h) / (d^{3/2}/h^{1/2}) = (d/h)^{1/2} \ll h^{-5/12}$,
and the factor $\exp(-d^2/(2h)) = 1 - \Theta(d^2/h)$ is much closer to $1$
than the factor $1 - \Theta(d^{3/2}/h^{1/2})$
if $h$ is large.
}
\end{remark}

\begin{remark}		\label{remark:hd3}
{\rm
Corollary~\ref{cor:Chebyshev_Lovasz} gives a nontrivial lower bound if
$h > 2\pi e d^2(d-1)$.  This is a weaker condition than the condition
$h > \pi d^4/2$ of Corollary~\ref{cor:lower_bd_via_Chebyshev}
(see Remark~\ref{remark:lower_bd_via_Chebyshev}) if $d \ge 10$.
For $2 \le d \le 9$, Corollary~\ref{cor:lower_bd_via_Chebyshev} is
sharper than Corollary~\ref{cor:Chebyshev_Lovasz}.
}
\end{remark}

We can improve on Theorem~\ref{thm:Chebyshev_Lovasz} and
Corollary~\ref{cor:Chebyshev_Lovasz} by treating the diagonal
and off-diagonal elements of $G$ differently.  For the diagonal
elements we can use Cantelli's inequality since ``large'' diagonal elements
are harmless~-- only ``small'' diagonal elements are ``bad''.
For the off-diagonal elements we can use
Hoeffding's inequality, because each off-diagonal element can be
written as a sum of independent random variables (this is not true
for the diagonal elements). 
The Lov\'asz Local Lemma can be applied much as
in the proof of Theorem~\ref{thm:Chebyshev_Lovasz}.
To handle the different bounds on diagonal and off-diagonal elements 
we need Lemma~\ref{lemma:super_duper_perturbation_bound}.
The parameters $\lambda$ and $t$ are chosen so that the probability
of an off-diagonal element being ``bad'' is the same as
the probability of a diagonal element being ``bad'' (more precisely,
our upper bounds on these probabilities are the same).
This choice is not optimal, but simplifies the application of the
Lov\'asz Local Lemma, since we can use the symmetric case
of the Lemma.
For a choice of $\lambda$ and $t$ giving unequal probabilities
(but not using the Lov\'asz Local Lemma),
see Theorem~\ref{thm:two_parameter_bound}.

\begin{theorem}		\label{thm:Cantelli_Hoeffding_Lovasz}
Suppose $d \ge 2$, $4 \le h \in \cal H$, $n = h+d$,
$\lambda 
 = (4e(d-1)-1)^{1/2}\sigma/\mu$,
$t 
 = (2\ln(8e(d-1)))^{1/2}/\mu$,
and $G$ as in~$\S\ref{subsec:properties}$.
If $\lambda + (d-1)t \le 1$,
then with positive probability we have
\begin{equation}
\frac{\det G}{\mu^d} \ge (1 - \lambda-(d-1)t)(1-\lambda+t)^{d-1}.
	\label{eq:two_parameter_boundA}
\end{equation}
\end{theorem}

\begin{proof}
Define $M := \mu^{-1}G$.
For the purposes of this proof we 
say that a diagonal element $m_{ii}$ of $M$ is \emph{bad}
if $m_{ii} < 1-\lambda$ (note the one-sided constraint); 
otherwise $m_{ii}$ is \emph{good} (so a good $m_{ii}$ can be large, but
not too small).
We say that an off-diagonal element $m_{ij} \;\; (i\ne j)$
is \emph{bad} if $|m_{ij}| > t$; otherwise $m_{ij}$ is \emph{good}.
We say that $G$ is \emph{good} if all the elements of $M$ are good;
otherwise $G$ is \emph{bad}.
If $G$ is good, then the conditions 
of Lemma~\ref{lemma:super_duper_perturbation_bound} apply to $M$ with
$(\delta,\ve) = (\lambda, t)$.

Define $p := 1/(4e(d-1))$ and $\tau := \sigma/\mu$, so
$\lambda = (1/p-1)^{1/2}\tau$.
By Cantelli's inequality, the probability that a diagonal
element $m_{ii}$ is bad is
\[\Prob[m_{ii} < 1-\lambda] \le \frac{\tau^2}{\tau^2 + \lambda^2} = p\,.\]
We can apply Hoeffding's inequality to the off-diagonal elements
$m_{ij}$ ($i \ne j$) since equation~\eqref{eq:sum_indep_vars} shows that,
in the case $i=1$ (which we consider without loss of generality),
$m_{1j}\; (= f_{1j})$ 
for $1 < j \le d$ is a sum of $h$ independent random variables
$u_{1k}b_{kj}$, where the elements $b_{kj}$ ($1 \le k \le h$)
of the $j$-th column of $B$ are distributed
independently and randomly in $\{\pm1\}$, and the multipliers
$u_{1k}$, which may be regarded as constants since they are independent%
\footnote{They are not independent of the first column of $B$, which is
why the argument does not apply to $m_{11}$ (or $f_{11}$).  
Similarly, the argument does
not apply to other diagonal elements $m_{ii}$ (or $f_{ii}$).}
of the $j$-th column of $B$, satisfy
$\sum_{k=1}^h u_{1k}^2 = 1$ in view of Lemma~\ref{lemma:u_sum}.
Thus $m_{1j}$ is a sum of $h$ independent, bounded random variables, with
bounds $[-|u_{1k}|,+|u_{1k}|]$ ($1 \le k \le h$).
It follows that, by Hoeffding's inequality (Proposition~\ref{prop:Hoeffding}),
the probability that an off-diagonal element
$m_{ij}$ ($i \ne j$) is bad is
\[\Prob[|m_{ij}| > t] \le 2\exp(-\mu^2t^2/2) = p\,.\]
{From} Lemma~\ref{lem:independence},
each $m_{ij}$ depends on at most $4d-4$ of the $m_{k\ell}$,
and it follows from the Lov\'asz Local Lemma
(Proposition~\ref{prop:Lovasz} with $D = 4(d-1)-1$)
and the definition of $p$ that
$\Prob[G\; \text{is good}] > 0$.
Thus, from Lemma~\ref{lemma:super_duper_perturbation_bound}, 
with positive probability we have
\[\det M \ge (1 - \lambda - (d-1)t)(1-\lambda+t)^{d-1}\,.\]
Since $\det G = \mu^d\det M$, this completes the proof.
\end{proof}

\begin{remark}			\label{remark:always_applicable}
{\rm
The condition $\lambda + (d-1)t \le 1$ is equivalent to
\begin{equation}	\label{eq:necessary_cond}
\mu \ge (4e(d-1)-1)^{1/2}\sigma + (d-1)(2\ln(8e(d-1)))^{1/2}\,,
\end{equation}
but $\mu > (2h/\pi)^{1/2}$,
so the condition is satisfied if 
$(d^2\ln d)/h$ is sufficiently small.
A simple sufficient condition is
\begin{equation} 
h \ge \pi d^2(4+\ln d)\,.	\label{eq:simple_cond}
\end{equation}
This can be proved using the
inequalities $\mu > (2h/\pi)^{1/2}$ and $\sigma \le 1/4$; we omit the
details. 
By results of Craigen~\cite{Craigen1} or Livinskyi~\cite{Livinskyi},
the inequality \eqref{eq:simple_cond}
(and hence also~\eqref{eq:necessary_cond}) 
holds for all sufficiently large~$n$ (assuming, as always, that
$h$ is maximal and $d$ minimal with $h+d=n$).
}
\end{remark}

\begin{remark}		\label{remark:thm4}
{\rm
In the proof of Theorem~\ref{thm:lower_bd_via_Chebyshev} we did not use
the Lov\'asz Local Lemma, and we obtained a sharper result than
that of Theorem~\ref{thm:Chebyshev_Lovasz} for $d < 10$ (see
Remark~\ref{remark:hd3}). Similarly, we can improve
Theorem~\ref{thm:Cantelli_Hoeffding_Lovasz} by not using the Lov\'asz
Local Lemma for small~$d$.  Instead of taking $p = 1/(4e(d-1))$ we take
$p = 1/d^2 - \ve$, and later let $\ve \to 0$. In this way we obtain
the inequality~\eqref{eq:two_parameter_boundA} with
$\lambda = (d^2-1)^{1/2}\sigma/\mu$,
$t = (2\ln(2d^2))^{1/2}/\mu$,
which is an improvement on Theorem~\ref{thm:Cantelli_Hoeffding_Lovasz}
for $d < 10$.
}
\end{remark}

\begin{remark}		\label{remark:one_or_the_other}
{\rm
At least one  of Theorem~\ref{thm:small_d} or
Theorem~\ref{thm:Cantelli_Hoeffding_Lovasz} is always
applicable.
In the region $n < 668$ the Hadamard conjecture has been verified,
so $d \le 3$ and Theorem~\ref{thm:small_d} applies.  
Consider the complementary region $n \ge 668$.
For $1 \le d \le 6$ the condition~\eqref{eq:simple_cond} holds.  
For $d \ge 7$ the condition~\eqref{eq:simple_cond} 
is weaker than the condition $h \ge 6d^3$ considered in~\cite{rpb253}.
Thus, it is sufficient to check the $13$ cases
$(h, d) = (h, h'-h+1)$, where the exceptional intervals $(h,h')$ are
listed in~\cite[Table~1]{rpb253}. 
We find numerically that
condition~\eqref{eq:necessary_cond} holds for all of these.
For example, the first entry with $(h,h') = (664,672)$ is covered 
as the right side of~\eqref{eq:necessary_cond} is $19.09\ldots$
but $\mu(664) = 21.55\ldots > 19.09$.
Thus Theorem~\ref{thm:Cantelli_Hoeffding_Lovasz} is always
applicable for $d \ge 4$.
}
\end{remark}

\pagebreak[3]

\begin{corollary} 	\label{cor:Cantelli_Hoeffding_Lovasz}
Under the conditions of Theorem~$\ref{thm:Cantelli_Hoeffding_Lovasz}$, 
we have
\begin{equation}	\label{eq:cor4bd1}
R(n) \ge \left(\frac{2}{\pi e}\right)^{d/2}
  (1 - \lambda-(d-1)t)(1-\lambda+t)^{d-1}e^{-d^2/(2h)}\,.
\end{equation}	
For $d = o(h/\log h)^{1/2}$ 
this gives
\begin{equation}	\label{eq:cor4bd2}
R(n) \ge \left(\frac{2}{\pi e}\right)^{d/2}
   \exp\left(O(d^{3/2}/h^{1/2})\right)\,.
\end{equation}
\end{corollary}

\begin{remark}
{\rm
Corollary~\ref{cor:Cantelli_Hoeffding_Lovasz} gives a nontrivial bound
for $d \ll h^{1/2 - \ve}$,
whereas Corollary~\ref{cor:Chebyshev_Lovasz} gives nothing useful
if $d \gg h^{1/3+\ve}$.
Also, the constant implied by the ``$O$'' notation is smaller for
Corollary~\ref{cor:Cantelli_Hoeffding_Lovasz}
than for Corollary~\ref{cor:Chebyshev_Lovasz}~-- if both $d$ and $h$
are large, the implied constant in
Corollary~\ref{cor:Cantelli_Hoeffding_Lovasz} is
$(1-3/\pi)\sqrt{2\pi e} \approx 0.186$,
whereas in Corollary~\ref{cor:Chebyshev_Lovasz} the
corresponding constant is $\sqrt{2\pi e} \approx 4.13$.
}
\end{remark}

\begin{proof}[Proof of Corollary~$\ref{cor:Cantelli_Hoeffding_Lovasz}$]
The proof of the inequality~\eqref{eq:cor4bd1}
is similar to the proof of~\eqref{eq:C1bd1}--\eqref{eq:C1f} above,  
using the bound of Theorem~\ref{thm:Cantelli_Hoeffding_Lovasz}
instead of the bound of Theorem~\ref{thm:lower_bd_via_Chebyshev}.

To prove~\eqref{eq:cor4bd2}, it is sufficient to show that
\begin{equation}	\label{eq:d22h}
(1 - \lambda-(d-1)t)(1-\lambda+t)^{d-1}e^{-d^2/(2h)}
 = \exp\left(O(d^{3/2}/h^{1/2})\right)\,.
\end{equation}
Taking logarithms, and assuming for the moment that
\begin{equation}	\label{eq:cor4condit}
\lambda + (d-1)t \le 1/2,
\end{equation}
we see that~\eqref{eq:d22h} is equivalent to showing that
\[
\lambda + (d-1)t + (d-1)(\lambda - t) + \frac{d^2}{2h}
 = O(d^{3/2}/h^{1/2}),
\]
which simplifies to
\begin{equation}	\label{eq:cor4bd3}
d\lambda + \frac{d^2}{2h} = O(d^{3/2}/h^{1/2}).
\end{equation}
Using the definitions of $\lambda$ and $t$, and the facts that
$\sigma = O(1)$ and $\mu \sim ch^{1/2}$, we see that
$\lambda = O((d/h)^{1/2})$, $t = O((\ln d)^{1/2}/h^{1/2})$.
Thus, the dominant term on the left side of~\eqref{eq:cor4bd3}
is $d\lambda = O(d^{3/2}/h^{1/2})$,
and the condition~\eqref{eq:cor4condit} is satisfied (for sufficiently
large $h$) if $d(\ln d)^{1/2}/h^{1/2} = o(1)$. The latter condition
follows from the assumption $d = o(h/\log h)^{1/2}$.
\end{proof}

The following theorem gives asymptotically better results
than Theorem~\ref{thm:Cantelli_Hoeffding_Lovasz}.
The proof uses the independence of the diagonal elements $g_{ii}$
but does not use the Lov\'asz Local Lemma. It might be possible to
sharpen the inequality~\eqref{eq:bad_offdiag_prob} via the
Lov\'asz Local Lemma, but this would complicate the argument while
giving only a small improvement in the final result
(only the right-hand side of~\eqref{eq:nonempty-condit1} and the
function $L(d)$ defined by~\eqref{eq:Ld_defn} would change).

\begin{theorem}		\label{thm:two_parameter_bound}
If $d \ge 2$, $4 \le h \in \cal H$, $n = h+d$,  
$\lambda \in (0,1)$, $t \ge 0$, $\tau = \sigma/\mu$,
$G$, $\sigma$ and $\mu$ as in~$\S\ref{subsec:properties}$, and
\begin{equation}
\exp(\mu^2t^2/2 - d\tau^2/\lambda^2) \ge 2d(d-1)\,,
	\label{eq:nonempty-condit1}
\end{equation}
then
\begin{equation}
\frac{\det G}{\mu^d} \ge (1-\lambda - (d-1)t)(1 - \lambda+t)^{d-1}
	\label{eq:two_parameter_bound}
\end{equation}
occurs with positive probability.
\end{theorem}
\begin{remark}
{\rm
The inequality~\eqref{eq:two_parameter_bound} is the same
as the inequality~\eqref{eq:two_parameter_boundA} occurring
in Theorem~\ref{thm:Cantelli_Hoeffding_Lovasz}, but the choice
of $(\lambda, t)$ is different in 
Theorem~\ref{thm:Cantelli_Hoeffding_Lovasz}.
}
\end{remark}
\begin{proof}[Proof of Theorem~$\ref{thm:two_parameter_bound}$]
For the purposes of this proof
we say that a matrix $G$ is \emph{good} if 
$g_{ii} \ge \mu(1-\lambda)$ and 
$|g_{ij}| \le \mu t$
for all $i \ne j$
(this definition is equivalent to the one used in the
proof of Theorem~\ref{thm:Cantelli_Hoeffding_Lovasz}).
{From} Cantelli's inequality we have
\[
\Prob[g_{ii} \ge \mu(1-\lambda)] \ge \frac{\lambda^2}{\tau^2 + \lambda^2}
 > \exp(-\tau^2/\lambda^2)\,.
\]
Since the $g_{ii}$ are independent, we deduce that
\begin{equation}
\Prob[\min \{g_{ii}: 1 \le i \le d\} \ge \mu(1-\lambda)]
 > \exp(-d\tau^2/\lambda^2)\,.		\label{eq:good_diag_prob}
\end{equation}
Also, as in the proof of Theorem~\ref{thm:Cantelli_Hoeffding_Lovasz},
Hoeffding's inequality implies that
\[\Prob[|g_{ij}| \ge \mu t] \le 2\exp(-\mu^2 t^2/2) \text{ for } i \ne j\,.\]
Thus
\begin{equation}
\Prob[\max \{|g_{ij}|: 1 \le i, j \le d, i \ne j\} \ge \mu t] \le 
	2d(d-1)\exp(-\mu^2 t^2/2)\,.	\label{eq:bad_offdiag_prob}
\end{equation}
{From} the inequalities~\eqref{eq:good_diag_prob} 
and~\eqref{eq:bad_offdiag_prob} we see\footnote{%
Informally, \eqref{eq:good_diag_prob} gives a lower bound on 
the probability that the diagonal
elements $g_{ii}$ are all good,
and~\eqref{eq:bad_offdiag_prob} gives an upper bound on the probability 
that at least one of the off-diagonal elements $g_{ij}$ is bad.
If the first probability exceeds the second, then a good $G$ occurs
with positive probability.}
that the condition
\begin{equation}
\exp(-d\tau^2/\lambda^2) \ge 2d(d-1)\exp(-\mu^2 t^2/2)
	\label{eq:nonempty-condit2}
\end{equation}
implies that, with positive probability, a random choice of $B$ gives
a good matrix~$G$.
Thus, \emph{some} choice of $B$ gives a good matrix~$G$.
However, the condition~\eqref{eq:nonempty-condit2}
is equivalent to the condition~\eqref{eq:nonempty-condit1}
in the statement of the Theorem.
To conclude the proof, it suffices to
observe that the lower bound~\eqref{eq:two_parameter_bound}
on $\det(G)$ for a good matrix~$G$ 
follows from Lemma~\ref{lemma:super_duper_perturbation_bound}
applied to $M = \mu^{-1}G$.
\end{proof}
\begin{remark}
{\rm
There are two parameters, $\lambda$ and $t$, occurring in
Theorem~\ref{thm:two_parameter_bound}.
In stating the Theorem we excluded the case $d=1$, because if $d=1$ 
then $t$ is irrelevant
and we may take $\lambda$ arbitrarily close to $0$, giving
$|\det(G)| \ge \mu^d\,,$
as obtained previously by Brown and Spencer~\cite{BS}
and (independently) by Best~\cite{Best}.
To obtain the best bound~\eqref{eq:two_parameter_bound} for $d \ge 2$ we
choose~$t$ so that equality holds in~\eqref{eq:nonempty-condit1}.
Thus, in the following, we choose
\begin{equation}
t^2 = \frac{2(d\tau^2/\lambda^2 + L(d))}{\mu^2}\,,	\label{eq:t}
\end{equation}
where
\begin{equation}		\label{eq:Ld_defn}
	L(d) := \ln[2d(d-1)].
\end{equation}
The optimal $\lambda \in (0,1)$ may be found by a straightforward numerical
optimisation.
In most cases there is no need for this,
as Corollary~\ref{cor:lambda_choice} gives
a result that is close to optimal.
}
\end{remark}

\begin{corollary}	\label{cor:lambda_choice}
Suppose that $1 < d = o(h^{2/5})$ and we choose
\begin{equation}
\lambda = \left(\frac{2d(d-1)\sigma^2}{\mu^4}\right)^{1/3}
	\label{eq:good_lambda}
\end{equation}
in Theorem~$\ref{thm:two_parameter_bound}$.
Then we obtain
\begin{equation}	\label{eq:cor5ineq}
{\cal R}(n) \ge \left(\frac{2}{\pi e}\right)^{d/2}
	\exp\left(d\sqrt{\frac{\pi}{2h}}-\frac{3d\lambda}{2}
	+ O((d\lambda)^{3/2})\right)\,.
\end{equation}
\end{corollary}

\begin{remark}	\label{remark:justify_choice_of_lambda}
{\rm
The choice~\eqref{eq:good_lambda} is motivated as follows.
When $\lambda$ and $t$ are small, 
the lower bound~\eqref{eq:two_parameter_bound} is\\
\[\det(G)/\mu^d \ge \exp\left(-d\lambda - \frac{d(d-1) t^2}{2} + 
	O(d^3t^3)\right)\,,\]
so to obtain a good lower bound on $\det(G)$
we should minimise
\[\lambda + \frac{(d-1)t^2}{2}\,\raisedot\]
Since $t^2$ is given by~\eqref{eq:t}, we should minimise
\[\lambda + (d-1)(d\tau^2/\lambda^2 + L(d))/\mu^2\,.\]
Since the term involving $L(d)$ is independent of $\lambda$,
we ignore it and minimize
\[f(\lambda) := \lambda + \frac{d(d-1) \tau^2}{\mu^2\lambda^2}\,.\]
Differentiating $f(\lambda)$
with respect to $\lambda$, setting $f'(\lambda) = 0$,
and using $\tau = \sigma/\mu$,
we obtain~\eqref{eq:good_lambda}.
Also, $\min_{x>0} f(x) = 3\lambda/2$,
where $\lambda$ is given by~\eqref{eq:good_lambda}.

Since $\mu \asymp h^{1/2}$ and $\sigma \asymp 1$,
we see that $\lambda \asymp(d/h)^{2/3}$.
Thus $d\lambda \asymp d^{5/3}/h^{2/3}$, which is asymptotically smaller than
the terms of order $d^{3/2}/h^{1/2}$ occurring in
Theorem~\ref{thm:Chebyshev_Lovasz} and
Corollaries~\ref{cor:Chebyshev_Lovasz}--\ref{cor:Cantelli_Hoeffding_Lovasz}.
Thus Corollary~\ref{cor:lambda_choice} is asymptotically sharper.
This is significant in the proof of Theorem~\ref{thm:almost_all_n} below.
}
\end{remark}
\begin{proof}[Proof of Corollary~$\ref{cor:lambda_choice}$]
Substitution of~\eqref{eq:good_lambda}
and~\eqref{eq:t} into the bound~\eqref{eq:two_parameter_bound},
then taking logarithms
and estimating the errors involved as
in Remark~\ref{remark:justify_choice_of_lambda}, shows that
\begin{equation}	\label{eq:cor5ineq1}
|\det G| \ge 
  \mu^d \exp\left(-\frac{3d\lambda}{2} + 
	O((d\lambda)^{3/2}) + O(d^2/h)\right).
\end{equation}
Now $d\lambda \asymp (d^5/h^2)^{1/3}$, 
so $d\lambda = o(1)$ iff $d = o(h^{2/5})$.
Also, from Lemma~\ref{lemma:sigma_asymptotics},
\begin{equation}	\label{eq:cor5mu}
\mu = \left(\frac{2h}{\pi}\right)^{1/2}
	\exp\left[\left(\frac{\pi}{2h}\right)^{1/2} + 
	  O\left(\frac{1}{h}\right)\right].
\end{equation}
Using the Schur complement lemma and Lemma~\ref{lemma:uncond2},
it follows from~\eqref{eq:cor5ineq1}--\eqref{eq:cor5mu} that
\[\Dbar(n) \ge \left(\frac{2}{\pi e}\right)^{d/2}\exp(\Delta),\]
where
\[
\Delta = d\sqrt{\frac{\pi}{2h}} - \frac{3d\lambda}{2} + 
	O((d\lambda)^{3/2}) + O(d^2/h).
\]
Since $d^2/h \ll (d\lambda)^{3/2} \asymp d^{5/2}/h$, 
the second ``$O$'' term can be subsumed by the first ``$O$'' term.
\end{proof}

We now extend 
Theorem~\ref{thm:small_d} 
to cases $d > 3$, provided that
$n$ is sufficiently large, where the threshold $n_0$ is independent of~$d$.
This improves Theorem~$1$ above, which assumes $d \le 3$.
It also improves Corollary~$2$ of~\cite{rpb253},
where $d > 3$ is allowed,
but the threshold is a rapidly-growing function of~$d$.
\begin{theorem}		\label{thm:almost_all_n}
Assume that $n = h+d$, where $d\ge0$, $h \in {\Had}$, and $d$ is minimal.
There exists an absolute constant $n_0$ such that, for all $n \ge n_0$,
\[\Dbar(n) \ge \left(\frac{2}{\pi e}\right)^{d/2}.\]
Moreover, the inequality is strict if $d > 0$.
\end{theorem}
\begin{proof}
For $d \le 3$ the result follows from Theorem~\ref{thm:small_d}, so
we may assume that $d \ge 4$.
In Corollary~\ref{cor:lambda_choice}, 
$\lambda \asymp (d/h)^{2/3}$.
Thus $\lambda h^{1/2} \asymp 
(d/h^{1/4})^{2/3}$.
{From}~\eqref{eq:gamma_Livinskyi},  
$d = o(h^{1/4})$, so $\lambda = o(h^{-1/2})$.
Thus, for sufficiently large $h$, the argument
$d(\sqrt{\pi/(2h)} - O(\lambda))$
of the exponential in~\eqref{eq:cor5ineq} is positive,
implying that $\Dbar(n) > (2/(\pi e))^{d/2}$.
\end{proof}
\begin{remark}
{\rm 
Using Corollary~$6$ of~\cite{rpb253},
which follows from Theorem~$5.4$ of Livinskyi~\cite{Livinskyi},
we can show that $n_0 = 10^{45}$ is sufficient in
Theorem~\ref{thm:almost_all_n}.
No doubt this value of $n_0$ can be
reduced considerably. Since this paper is long enough, we resist
the temptation to attempt any such reduction here. 
As mentioned in \S\ref{sec:intro},
we conjecture that Theorem~\ref{thm:almost_all_n} holds with $n_0 = 1$.
}
\end{remark}

\section{Numerical examples} 		\label{sec:numerical}

Consider the case $n=668$.
At the time of writing it is not known
whether a Hadamard matrix of this order exists.  Assuming it does not, we
take $h=664$, $d=4$, $n = h+d = 668$. 
Thus 
$\mu \approx 21.55231$, $\sigma^2 \approx 0.04638855$.
Column~$2$ of Table~\ref{tab:664} 
gives various lower bounds on $\det(G)/\mu^d$ (for $G$ that
occurs with positive probability).
These may be converted to lower bounds on $\Dbar(n)$ if desired;
the constant of proportionality is
$\mu^dh^{h/2}/n^{n/2} \approx 0.06583$.
We give $\det(G)/\mu^d$ as it is a useful ``figure of merit'' to compare
different probabilistic approaches~-- the upper limit of these
approaches is $\det(G)/\mu^d = 1$.

Column~$5$ of Table~\ref{tab:664} gives the corresponding bounds
for $d=7$, $n=671$.  This is a difficult case since it is the smallest
with $d=7$ (assuming as before that $664 \not\in\Had$).
Theorem~\ref{thm:lower_bd_via_Chebyshev} and
Remark~\ref{remark:improved_simple} give negative bounds
since $d^4/h$ is too large. Similarly for Theorem~\ref{thm:Chebyshev_Lovasz}
since $d^3/h$ is too large (even when $d=4$). However, 
Theorem~\ref{thm:Cantelli_Hoeffding_Lovasz} gives a useful bound
(in agreement with Remark~\ref{remark:one_or_the_other}),
as does Theorem~\ref{thm:two_parameter_bound}.

The entries in the rows labelled ``Corollary~\ref{cor:lambda_choice}''
use~\eqref{eq:good_lambda} to define~$\lambda$;
the entries in rows labelled ``Theorem~\ref{thm:two_parameter_bound}'' use
optimal values of $\lambda$; 
$t$~is defined by~\eqref{eq:t} in both cases.

\begin{table}[ht]       
\begin{center}          
\begin{tabular}{l|c|c|c||c|c|c|}
\cline{2-7}
& \multicolumn{3}{|c||}{$d=4$, $n=668$}
& \multicolumn{3}{|c|}{$d=7$, $n=671$}\\
\hline
\multicolumn{1}{|c|}{Bound}
	& $|G|/\mu^d$ & $\lambda$ & $t$
        & $|G|/\mu^d$ & $\lambda$ & $t$\\
\hline
\multicolumn{1}{|l|}
{Theorem \ref{thm:lower_bd_via_Chebyshev}}
 & $0.2576$ & --- & ---  & --- & --- & ---\\
\multicolumn{1}{|l|}
{Remark \ref{remark:improved_simple}}
 & $0.3521$ & --- & ---  & --- & --- & ---\\
\multicolumn{1}{|l|}
{Theorem \ref{thm:Cantelli_Hoeffding_Lovasz}}
 & $0.6781$ & $0.05619$ & $0.1341$
 & $0.0742$ & $0.08010$ & $0.1448$\\
\multicolumn{1}{|l|}
{Remark \ref{remark:thm4}}
 & $0.7565$ & $0.03870$ & $0.1222$
 & $0.1326$ & $0.06924$ & $0.1405$\\
\multicolumn{1}{|l|}
{Corollary \ref{cor:lambda_choice}}
 & $0.7975$ & $0.01728$ & $0.1394$
 & $0.1125$ & $0.02624$ & $0.1531$\\ 
\multicolumn{1}{|l|}
{Theorem \ref{thm:two_parameter_bound}}
 & $0.7990$ & $0.01937$ & $0.1352$
 & $0.1667$ & $0.04238$ & $0.1441$\\ 
\hline
\end{tabular}
\caption{Lower bounds for $h = 664$, $d\in\{4,7\}$, $n=h+d$.}\label{tab:664}
\end{center}
\end{table}

\pagebreak[3]
Table~\ref{tab:996} 
gives various lower bounds on $\det(G)/\mu^d$ for the
cases $h = 996$, $d\in\{2,3\}$, so $n\in\{998,999\}$.  
Here 
$\mu \approx 26.17449$ and
$\sigma^2 \approx 0.04594917$.
Lemma~\ref{lemma:expand_det}
is applicable, as $d \le 3$.
Lemma~\ref{lemma:expand_det} does not state an explicit bound for
$\det(G)/\mu^d$; Table~\ref{tab:996} gives the value
$1 - \eta/\mu^d$
that occurs in the proof
of Lemma~\ref{lemma:expand_det}~-- see the inequality~\eqref{eq:G_eta_bound}.
Since $\eta/\mu^d = O_d(h^{-1})$, it is not surprising
that Lemma~\ref{lemma:expand_det} gives the sharpest
bound for $d \le 3$.

\begin{table}[ht]        
\begin{center}
\begin{tabular}{l|c|c|c||c|c|c|}   
\cline{2-7}
& \multicolumn{3}{|c||}{$d=2$, $n=998$}
& \multicolumn{3}{|c|}{$d=3$, $n=999$}\\
\hline
\multicolumn{1}{|c|}{Bound}
        & $|G|/\mu^d$ & $\lambda$ & $t$
        & $|G|/\mu^d$ & $\lambda$ & $t$\\
\hline
\multicolumn{1}{|l|}
{Lemma \ref{lemma:expand_det}}
 & $0.9985$ & --- & ---
 & $0.9910$ & --- & ---\\
\multicolumn{1}{|l|}
{Theorem \ref{thm:lower_bd_via_Chebyshev}}
& $ 0.8472$ & --- & ---
 & $0.6562$ & --- & ---\\
\multicolumn{1}{|l|}
{Remark \ref{remark:improved_simple}}
 & $0.8895$ & --- & ---
 & $0.7160$ & --- & ---\\
\multicolumn{1}{|l|}
{Theorem \ref{thm:Chebyshev_Lovasz}}
 & $0.7480$ & --- & ---
 & $0.4655$ & --- & ---\\
\multicolumn{1}{|l|}
{Theorem \ref{thm:Cantelli_Hoeffding_Lovasz}}
 & $0.9402$ & $0.02573$ & $0.0948$
 & $0.8581$ & $0.03730$ & $0.1049$\\
\multicolumn{1}{|l|}
{Remark \ref{remark:thm4}}
 & $0.9658$ & $0.01418$ & $0.0779$
 & $0.9058$ & $0.02316$ & $0.0919$\\
\multicolumn{1}{|l|}
{Corollary \ref{cor:lambda_choice}}
 & $0.9741$ & $0.00732$ & $0.1066$
 & $0.9287$ & $0.01055$ & $0.1119$\\ 
\multicolumn{1}{|l|}
{Theorem \ref{thm:two_parameter_bound}}
 & $0.9741$ & $0.00733$ & $0.1065$
 & $0.9288$ & $0.01102$ & $0.1010$\\ 
\hline
\end{tabular}
\caption{Lower bounds for $h = 996$, $d\in\{2,3\}$, $n=h+d$.}\label{tab:996}
\end{center}
\end{table}

\pagebreak[4]

\end{document}